\documentclass[11pt,twoside]{article}

\usepackage{jmlr2e}
\usepackage[utf8]{inputenc}
\usepackage{color}
\usepackage{mathtools}
\usepackage{float}
\usepackage{url}
\usepackage[algo2e,ruled,linesnumbered,resetcount]{algorithm2e}
\usepackage{algorithm}
\usepackage{amsmath}
\usepackage{booktabs}
\usepackage{graphicx}
\usepackage{esint}
\PassOptionsToPackage{authoryear}{natbib}
\usepackage{amsbsy}
\usepackage{amstext}
\usepackage{bigints}
\usepackage{setspace}
\usepackage{bookmark}
\usepackage{breakcites}
\usepackage{breqn}
\usepackage{placeins}
\usepackage{color}
\usepackage{soul}

\usepackage{multirow}
\usepackage{makecell}

\usepackage{xcolor}

\newcommand{\mtp}[1]{\textcolor{orange}{\bf #1}}

\newtheorem{prop}{Proposition}
\newtheorem{lem}{Lemma}

\jmlrheading{1}{2019}{1-?}{05/19}{10/19}{LNP19}{Hengrui Luo, Giovanni Nattino and Matthew T. Pratola}

\ShortHeadings{Sparse Additive Gaussian Process Regression}{Luo, Nattino and Pratola}

\begin{document}
	
	
	\title{Sparse Additive Gaussian Process Regression}
	\author{\name{Hengrui Luo} \email luo.619@osu.edu \\
		\addr Department of Statistics\\
		The Ohio State University\\
		Columbus, OH 43210, USA \AND \name{Giovanni Nattino} \email nattino.1@osu.edu
		\\
		\addr Division of Biostatistics, College of Public Health\\
		The Ohio State University\\
		Columbus, OH 43210, USA \AND \name{Matthew T. Pratola} \email
		mpratola@stat.osu.edu \\
		\addr Department of Statistics\\
		The Ohio State University\\
		Columbus, OH 43210, USA\\
		$ $}
	
	\editor{}
	
	\maketitle


\begin{abstract}
In this paper we introduce a novel model for Gaussian process (GP)
regression in the fully Bayesian setting. Motivated by the ideas of
sparsification, localization and Bayesian additive modeling, our model
is built around a recursive partitioning (RP) scheme. Within each
RP partition, a sparse GP (SGP) regression model is fitted. A Bayesian additive
framework then combines multiple layers of partitioned SGPs, capturing
 both global trends and local refinements with efficient computations. The model
addresses both the problem of efficiency in fitting a full Gaussian
process regression model and the problem of prediction performance
associated with a single SGP. Our approach mitigates
the issue of pseudo-input selection and avoids the need for complex
inter-block correlations in existing methods.  The
crucial trade-off becomes choosing between many simpler local model
components or fewer complex global model components, which the practitioner can
sensibly tune. Implementation is via a 
Metropolis-Hasting Markov chain Monte-Carlo algorithm with Bayesian back-fitting.
We compare our model against popular alternatives on simulated and
real datasets, and find the performance is competitive, while the
fully Bayesian procedure enables the quantification of model uncertainties. 
\end{abstract}

\newpage{}

\section{\label{sec:Introduction}Introduction}

Gaussian process (GP) regression is a widely adopted regression model \citep{Rasmuessen&Williams2006}.
Taking a Bayesian approach, its posterior distribution provides a principled way to
quantify uncertainties while having nice theoretical properties \citep{Gelman2003}.
However, the computational cost of GP likelihood evaluations based
on an observed dataset $\{y,\mathcal{X}\}$ of size $n$ is of order
$\mathcal{O}(n^{3})$, which primarily results from the need to invert
an $n\times n$ covariance matrix. Therefore, the computational cost
could be prohibitively high in scenarios where large datasets need
to be analyzed. It is a focus of much current research to solve this
problem of high computational cost for GP regression \citep{Banerjee2012,Liu2018}.

Many approaches to circumvent this problem have been explored, such as
low-rank covariance approximation \citep{Titsias2009}, model likelihood
approximations \citep{KSN2008} and local GP approximations \citep{Snelson&Ghahramani2007,Gramacy&Apley2015}.
However, most of these approaches are not fully Bayesian.

We are inspired by the idea of low-rank sparse GP regression \citep{Snelson&Ghahramani2006}
and localization ideas \citep{Chipman1998,Lee_etal2017,Gramacy&Apley2015,Park&Huang2016,Nguyen-Tuong2009,Chipman2016,Lee_etal2017},
but we still want to incorporate these methods within a fully Bayesian
framework. Borrowing the framework of Bayesian (generalized) additive
modeling \citep{Hastie&Tibshirani1990,Hastie&Tibshirani2000}, 
we propose the Sparse Additive
Gaussian Process (SAGP) model.  SAGP combines sparse GP regression
and a recursive partition (RP) scheme within a fully Bayesian model.
It turns out that our approach can simultaneously handle both local and global features
 in large datasets while realizing gains in computational
efficiency. Furthermore, it provides principled uncertainty quantification
for parameters and posterior predictions.  A key feature of the approach is a much simplified fixed partitioning scheme that avoids the added computational costs of stochastic tree-based partitioning models (e.g. \citep{Chipman1998,Chipman2016,gramacy2007tgp}).
To our best knowledge, this kind of additive Bayesian model, combining
both sparsification and localization, has never been explored.

The organization of the paper is as follows. In section \ref{sec:Backgrounds}
we will briefly review the background knowledge for sparse GP, localization
and Bayesian additive modeling as they are essential ingredients of
SAGP modeling. In section \ref{sec:Sparse Additive Gaussian Process SAGP}
we will specify the SAGP model. Sections \ref{sec:Simulation Study}
and \ref{sec:Real-world-Data-Applications} are analyses 
of simulated and real-world datasets. Finally,
we conclude our paper with a discussion in section \ref{sec:Conclusion}.

\section{\label{sec:Backgrounds}Background}

\subsection{\label{sec:Gaussian Process Regression}Gaussian Process Regression}

We start with GP regression on the input domain $\mathsf{X}$ and
use the notation $N_{d}(\boldsymbol{m},\Sigma)$ to denote the $d$-dimensional
Gaussian distribution with mean vector $\boldsymbol{m}$ and covariance
matrix $\Sigma,$ and the notation $N_{d}(\boldsymbol{y}\mid\boldsymbol{m},\Sigma)$
to denote the $d$-dimensional Normal density evaluated at $\boldsymbol{y}\in\mathbb{R}^{n}$.
The prior of the mean regression function is assumed to be a GP with
known mean and covariance kernel function. Posterior estimation
and prediction arise from combining the prior
belief with the information contained in the likelihood of response
variables $\boldsymbol{y}=(y_{1},\ldots,y_{n})^{T},$ observed at known
input locations $\mathcal{X}=\{\boldsymbol{x}_{i}\in\mathbb{R}^{d},i=1,\ldots,n\}\subset\mathsf{X},$
by using Bayes theorem. We also call $f$ the target and the variable
$\boldsymbol{x}_{i}$ the input, based on the model form 
\begin{align}
\begin{array}{c}
y(\boldsymbol{x}_{i})=f(\boldsymbol{x}_{i})+\epsilon_{i},i=1,\ldots,n\\
\epsilon_{i}\sim N_{1}(0,\sigma_{\epsilon}^{2})
\end{array}
\end{align}
which expresses the relationship between input $\boldsymbol{x}_{i}$
and the unknown response $f(\boldsymbol{x}_{i})$ observed as $y_{i}$
with observational error $\epsilon_{i}$ having the variance $\sigma_{\epsilon}^{2}$.
Using vector notations we write $\boldsymbol{y}=(y_{1},\ldots,y_{n})^{T}=\left(y(\boldsymbol{x}_{1}),y(\boldsymbol{x}_{2}),\ldots,y(\boldsymbol{x}_{n})\right)^{T}$,
$\boldsymbol{f}=\left(f(\boldsymbol{x}_{1}),\ldots,f(\boldsymbol{x}_{n})\right)^{T}$
and the noise $\boldsymbol{\epsilon}\sim N_{n}(\boldsymbol{0}_{n},\sigma_{\epsilon}^{2}\boldsymbol{I}_{n})$
to yield $\boldsymbol{y}=\boldsymbol{f}+\boldsymbol{\epsilon}$.

Without loss of generality, it is often convenient to assume that the mean vector
$\boldsymbol{f}$ is a
realization of a zero mean Gaussian process, $\boldsymbol{f}\sim N_{n}(\boldsymbol{0},\boldsymbol{K}_{n})$,
where $\boldsymbol{K}_{n}=\left[K(\boldsymbol{x}_{i},\boldsymbol{x}_{j})\right]_{i,j=1}^{n},$
 with covariance kernel $K(\cdot,\cdot):\mathbb{R}^{d}\times\mathbb{R}^{d}\rightarrow\mathbb{R}$
encoding assumed properties of the unknown function $f$ to satisfy
the application of interest \citep{Rasmuessen&Williams2006}.

\subsection{\label{sec:Sparsification}Sparsification of Gaussian Processes}

There are a variety of sparse approximation approaches to GP regression
(e.g. \citealt{Lawrence_etal2003,QCandela&Rasmussen2005}).
A popular approach is the pseudo-input
(or latent variable) approach. By replacing the exact covariance matrix in the likelihood
computation with a low-rank approximation, one can greatly reduce
computational cost.  \citet{Snelson&Ghahramani2006} propose the Sparse Gaussian Process (SGP) model by using
a subset of the full inputs $\mathcal{X}=\{\boldsymbol{x}_{1},\ldots,\boldsymbol{x}_{n}\}$
as pseudo-inputs, denoted as $\bar{\mathcal{X}}=\{\bar{\boldsymbol{x}}_{1},\ldots,\bar{\boldsymbol{x}}_{m}\}\subset\mathcal{X}$,
for $m\ll n$. 
Then, $\bar{\boldsymbol{f}}=\left(f(\bar{\boldsymbol{x}}_{1}),f(\bar{\boldsymbol{x}}_{2}),\ldots,f(\bar{\boldsymbol{x}}_{m})\right)^{T}$
are called pseudo-targets, and 
\begin{align*}
\boldsymbol{K}_{n} & \coloneqq\left[K(\boldsymbol{x}_{k},\boldsymbol{x}_{l})\right]_{k,l=1}^{n},\\
\boldsymbol{K}_{m} & \coloneqq\left[K(\bar{\boldsymbol{x}}_{k},\bar{\boldsymbol{x}}_{l})\right]_{k,l=1}^{m},\\
\boldsymbol{K}_{nm} & =\left[K(\boldsymbol{x}_{i},\bar{\boldsymbol{x}}_{j})\right]_{i,j=1}^{n,m}=\boldsymbol{K}_{mn}^{T}
\end{align*}
denote the (cross-)covariance matrices among and between the full targets $\boldsymbol{f}$
and pseudo-targets $\bar{\boldsymbol{f}}$. Their approach treats the pseudo-inputs as (hyper-)parameters, resulting in a likelihood function that only requires the inversion of the dense $m\times m$ matrix $K_m$, a significant computational savings. 
The posterior and posterior predictive distributions
 can then be written in closed form by Gaussian conjugacy \citep{Snelson&Ghahramani2006}. 
 
For an SGP model with $m$ pseudo-inputs,
the full likelihood is $P(\boldsymbol{y}\mid\mathcal{X},\bar{\mathcal{X}},\bar{\boldsymbol{f}},\sigma^2_{\epsilon})=N_{n}(\boldsymbol{y}\mid\boldsymbol{K}_{nm}\boldsymbol{K}_{m}^{-1}\bar{\boldsymbol{f}},\boldsymbol{\Lambda}+\sigma_{\epsilon}^{2}\boldsymbol{I}_{n})$
where $\boldsymbol{\Lambda}=\text{diag}\left(K(\boldsymbol{x}_{i},\boldsymbol{x}_{i})-\boldsymbol{k}_{i}^{T}\boldsymbol{K}_{m}^{-1}\boldsymbol{k}_{i}\right)_{i=1}^{n}$ where $\boldsymbol{k}_{i}=\left(K(\bar{\boldsymbol{x}}_{1},\boldsymbol{x}_{i}),\ldots,K(\bar{\boldsymbol{x}}_{m},\boldsymbol{x}_{i})\right)^{T}$.
Using Bayes theorem, we can write the posterior distribution of pseudo-targets
as $P(\left.\bar{\boldsymbol{f}}\right|\mathcal{X},\boldsymbol{y},\bar{\mathcal{X}},\sigma^2_{\epsilon})=N_{m}(\bar{\boldsymbol{f}}\mid\boldsymbol{K}_{m}\boldsymbol{Q}_{m}^{-1}\boldsymbol{K}_{mn}\left(\boldsymbol{\Lambda}+\sigma_{\epsilon}^{2}\boldsymbol{I}\right)^{-1}\boldsymbol{y},\ \boldsymbol{K}_{m}\boldsymbol{Q}_{m}^{-1}\boldsymbol{K}_{m})$
where 
$\boldsymbol{Q}_{m}=\boldsymbol{K}_{m}+\boldsymbol{K}_{mn}\left(\boldsymbol{\Lambda}+\sigma^{2}_{\epsilon}\boldsymbol{I}\right)^{-1}\boldsymbol{K}_{nm}$.
The posterior predictive distribution for $y^{*}$ at a new input
$\boldsymbol{x}^{*}$, after integrating out the pseudo-target $\bar{\boldsymbol{f}}$,
can be written as $P(y_{*}\mid\boldsymbol{x}_{*},\mathcal{X},\boldsymbol{y},\sigma^2_{\epsilon})=N_{1}(\boldsymbol{k}_{*}^{T}\boldsymbol{Q}_{m}^{-1}\boldsymbol{K}_{mn}$
$\left(\boldsymbol{\Lambda}_{n}+\sigma_{\epsilon}^{2}\boldsymbol{I}_{n}\right)^{-1}\boldsymbol{y},\sigma_{\epsilon}^{2}+K_{**}-\boldsymbol{k}_{*}^{T}\boldsymbol{K}_{m}^{-1}\boldsymbol{k}_{*}+\boldsymbol{k}_{*}^{T}\boldsymbol{Q}_{m_{j}}^{-1}\boldsymbol{k}_{*})$,
where $\boldsymbol{k}_{*}=\left(K(\bar{\boldsymbol{x}}_{1},\boldsymbol{x}_{*}),\ldots,K(\bar{\boldsymbol{x}}_{m},\boldsymbol{x}_{*})\right)^{T}$.
In particular, when $n=m$ we obtain the posterior distributions
of the full Gaussian process model.

One central problem of the SGP approach is that the sparsification
depends on the choice of the pseudo-inputs $\bar{\mathcal{X}}$, which are treated as hyperparameters to be (somehow) selected once and then held fixed.
In the original work, \citet{Snelson&Ghahramani2006} propose to
choose the pseudo-inputs by optimizing the marginal likelihood,
or the KL divergence \citep{Titsias2009,Damianou&Lawrence2013}. In
our Bayesian approach, instead of using a fixed choice of pseudo-inputs
\citep{Titsias2009,Lee_etal2017}, we instead draw the pseudo-inputs from a prior.

\subsection{\label{sec:Bayesian Additive Modeling}Bayesian Additive Modeling
and Back-fitting}

Bayesian additive modeling \citep{Hastie&Tibshirani1990,Chipman1998}
is a flexible modeling technique that is widely adopted. Such 
additive models are formed by taking the sum of many model components, where each model component captures only a portion of the
overall response variability. In the Gaussian setting, fitting a Bayesian additive
model can be accomplished by using partial residuals to update each component sequentially in a so-called 
 back-fitting scheme \citep{Hastie&Tibshirani2000}. Following this scheme, we can represent an additive model with
$N$ components without intercept term in vector form as $\boldsymbol{y}=\sum_{j=1}^{N}\boldsymbol{f}_{j}+\boldsymbol{\epsilon},\boldsymbol{\epsilon}\sim N_{n}(\boldsymbol{0}_{n},\sigma_{\epsilon}^{2}\boldsymbol{I}_{n}).$

Bayesian back-fitting proceeds by fitting  each additive component, $\boldsymbol{f}_{j},$
by using the ``$j$-th partial residuals'', $\boldsymbol{r}_{j}=\boldsymbol{y}-\sum_{i\neq j}\boldsymbol{f}_{i}$.
These residuals are used as ``data'' for the $j$-th component. Starting
with a particular initial value, the back-fitting algorithm (Algorithm
3.1 in \citet{Hastie&Tibshirani2000}) iterates until the joint distribution
of all mean functions $\boldsymbol{f}_{1},\boldsymbol{f}_{2},\ldots,\boldsymbol{f}_{N}$
stabilizes.

One insight into the usefulness of this algorithm is to recognize that it allows updating the $f_j$'s one at a time rather than requiring an expensive joint update like the full GP regression on a large dataset.  Therefore, it would be advantageous to make the $f_j$ updates computationally cheap.

\subsection{\label{sec:Localization Partition Scheme}Localization via Partition
Schemes}

Partitioning the input space $\mathsf{X}=\bigcup_{j}\mathsf{X}_j$ has been another popular way of scaling-up regression models.
In this line of research, pioneering works were performed by \citet{Breiman1984},
\citet{denison1998bayesian} and \citet{Chipman1998,Chipman_etal2010,Chipman2016}.
Furthermore, various choices of partition schemes of the input domain
are discussed in the local GP regression literature \citep{Nguyen-Tuong2009, Gramacy&Apley2015,Park&Huang2016}.

In terms of Bayesian additive modeling, \citet{Chipman1998} model the
data in each partition $\mathsf{X}_j$ using an independent model component, conditional on the partitioning defined by a binary tree. This associates the fitted mean function (or target) $\boldsymbol{f}_{j}$
with the data lying in the specific partition $\mathsf{X}_j$. Subsequent works \citep{Gramacy&Apley2015,Chipman2016,Pratola2017}
 demonstrate that assembling many simpler
models over such partitioning schemes can usually out-perform a single
complex model fitted to the entire modeling domain.

As pointed out in \citet{Gramacy&Apley2015} and \citet{Park&Apley2018},
such localization of the input-domain will fit and predict non-stationary
datasets better. Also, multi-scale features of a dataset can usually
be well captured by introducing a hierarchical structure on the input
domain \citep{Fox&Dunson2012,Lee_etal2017}.

In our approach, capturing global and local features is accomplished
through a fixed partition scheme informed by the data $\mathcal{X}$. We will show how this can be done so that the partition scheme is well suited to the Bayesian backfitting algorithm, and use sparse model components to further enhance the scalability of the model.

\section{\label{sec:Sparse Additive Gaussian Process SAGP}Sparse Additive
Gaussian Process Regression (SAGP)}

The proposed SAGP model combines the three key ingredients of sparsification, Bayesian additive modeling (via backfitting), and localization in a clever way.  In particular, our model has the usual additive form,
\begin{align}
\boldsymbol{y} & =\sum_{j=1}^{N}\boldsymbol{f}_{j}+\boldsymbol{\epsilon},\label{eq:SAGP model}
\end{align}
for some error component $\boldsymbol{\epsilon}$ with variance $\sigma^2_\epsilon.$  Much effort in statistical modeling focuses on the $\boldsymbol{f}_{j}.$  For instance, in linear regression, $\boldsymbol{f}_{j}=\boldsymbol{X}_j\beta_j$ for the $j$th column of some design matrix $\boldsymbol{X}$ and vector parameter $\boldsymbol{\beta}$.  In our approach, each $\boldsymbol{f}_{j}$ has entries which are formed by weighted linear combinations of the pseudo-targets, $\boldsymbol{W}^T\bar{\boldsymbol{f}}_{j},$ and each vector of pseudo-targets $\bar{\boldsymbol{f}}_{j}$ arise from the pseudo-inputs $\bar{\mathcal{X}}^{(j)}$ belonging to the $j$th subdomain of the input domain $\mathsf{X}.$  Additional parameters $\boldsymbol{\kappa}$ will be involved in each component in forming the weights $\boldsymbol{W}$.  Finally, the subdomains are defined by a partitioning scheme, $\mathcal{B}_N$, and let  the collection of pseudo-inputs belonging to each partition be  $\bar{\mathcal{X}}^{(1)},\ldots,\bar{\mathcal{X}}^{(N)}.$  Then, our model takes a hierarchical form involving the likelihood function
$L(\boldsymbol{y}\vert \bar{f}_1,\ldots,\bar{f}_N,\boldsymbol{\kappa},\sigma^2_\epsilon,\mathcal{B}_N,\bar{\mathcal{X}}^{(1)},\dots,\bar{\mathcal{X}}^{(N)},\mathcal{X})$ as well as the prior distributions of the various additive model components in the overall model, $P(\bar{f}_1,\ldots,\bar{f}_N\vert\boldsymbol{\kappa},\mathcal{B}_N,\bar{\mathcal{X}}^{(1)},\ldots,\bar{\mathcal{X}}^{(N)},\mathcal{X}),$ $P(\boldsymbol{\kappa}\mid \mathcal{B}_N,\mathcal{X}),$  
$P(\sigma^2_\epsilon),$ and $P(\bar{\mathcal{X}}^{(1)},\ldots,\bar{\mathcal{X}}^{(N)}\vert\mathcal{B}_N,\mathcal{X}).$

To perform inference and prediction, we will be interested in the marginal posterior distribution $P(\bar{f}_1,\ldots,\bar{f}_N,\boldsymbol{\kappa},\sigma^2_{\epsilon}\vert \boldsymbol{y},\mathcal{B}_N,\mathcal{X}).$
Note that the posterior is dependent on the partitioning scheme $\mathcal{B}_N$ since it is held fixed in our modeling approach.
Therefore, we will start by describing the proposed partitioning scheme.  The partitioning scheme reduces computational cost by limiting the sample size in each model component by exploiting localization.  Second, conditional on this localization scheme, the model components (i.e. the $f_j$'s) themselves leverage the sparse Gaussian Process.  This sparsification reduces computational cost as described earlier.  Finally, our overall model combines all of these sparse localized components into a Bayesian additive model as defined by the likelihood function, and the overall model can be efficiently fit using Bayesian back-fitting.

\subsection{\label{sec:Definitions-and-Examples (RP scheme)}A Recursive Partitioning
Scheme}

We consider a recursive partitioning of the domain $\mathsf{X}$ that can be represented as a $2^d$-ary tree. Each node of the tree
corresponds to a subregion $B_{j}$ of $\mathsf{X}\subset\mathbb{R}^{d}$
called a \emph{block}. Only the node at the first level, i.e., the
root of the tree, corresponds to the whole domain ($B_{1}=\mathsf{X}$).
The collection of blocks corresponding to nodes at the same depth
of the tree is referred to as a \emph{layer}. The collection of all blocks across
all layers of the tree comprise the partitioning of  $\mathsf{X}.$  More formally, we define
a \emph{Recursive Partitioning} (RP) scheme as a collection of blocks
$\left\{ B_{1},...,B_{N}\right\}$ and layers $\mathcal{L}_{1},\ldots,\mathcal{L}_{L}$
of these blocks satisfying the following properties: 
\begin{enumerate}
\item[P1.] (Nestedness) For a block $B_{i}\subset\mathbb{R}^{d}$ in the $j$-th
layer $\mathcal{L}_{j}$, there exists a unique block $B_{k}\in\mathcal{L}_{j-1}\subset\mathbb{R}^{d}$
in the $(j-1)$-layer $\mathcal{L}_{j-1}$ such that $B_{i}\subset B_{k}$. 
\item[P2.] (Disjointedness, or non-overlapping) For two blocks $B_{i},B_{k}$
in the $j$-th layer $\mathcal{L}_{j}$ such that $B_{i}\neq B_{k}$,
their interiors do not intersect. 
\end{enumerate}
\FloatBarrier 

\begin{figure}

\begin{minipage}[c]{0.4\columnwidth}%
\includegraphics[width=8.5cm]{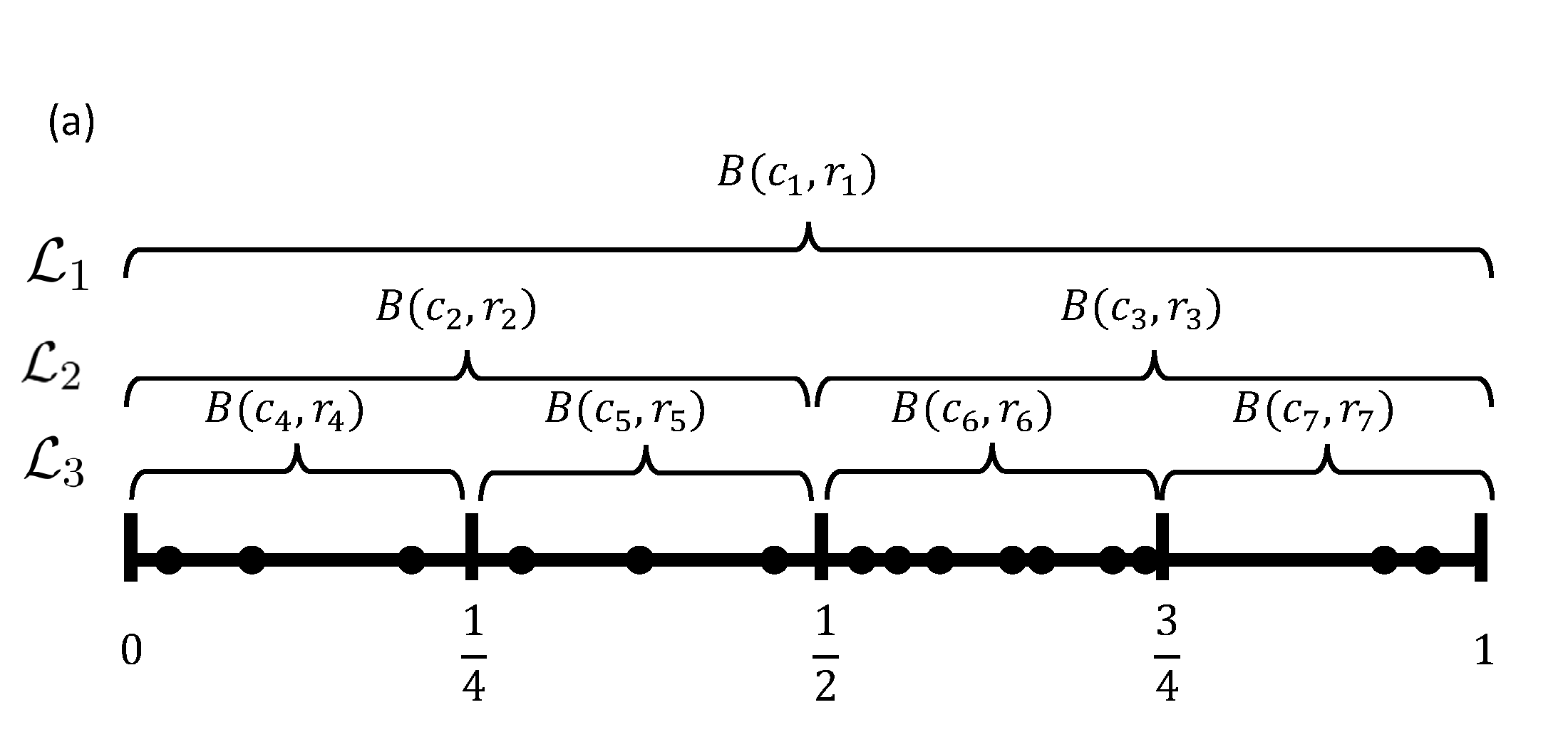}\\
\includegraphics[width=8.5cm]{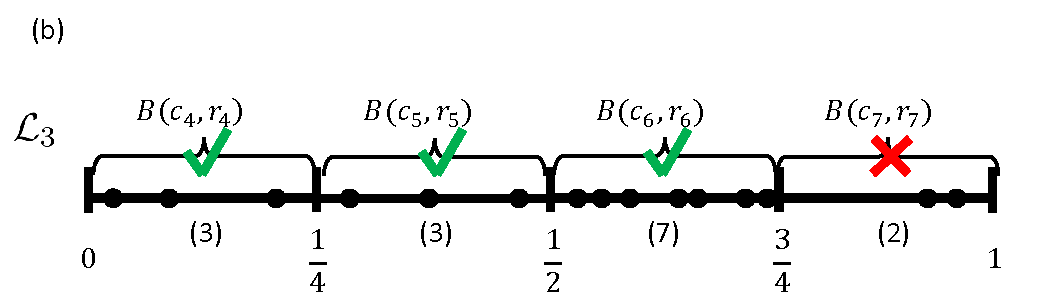}\\
\includegraphics[width=8.5cm]{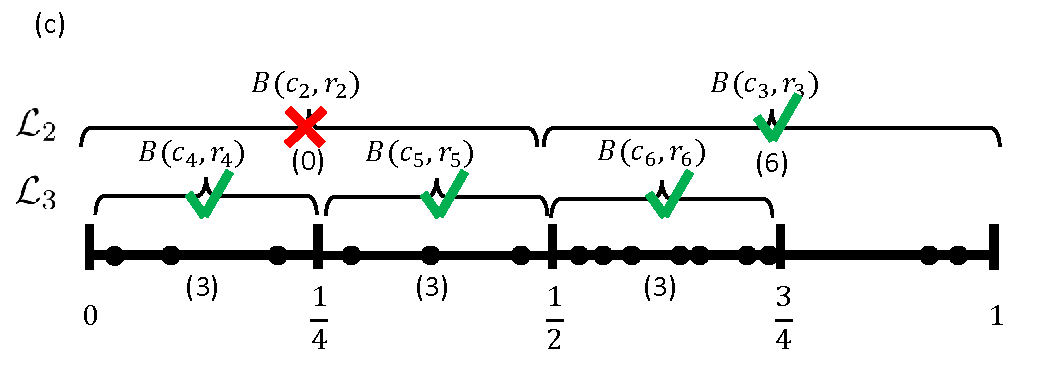}\\
\includegraphics[width=8.5cm]{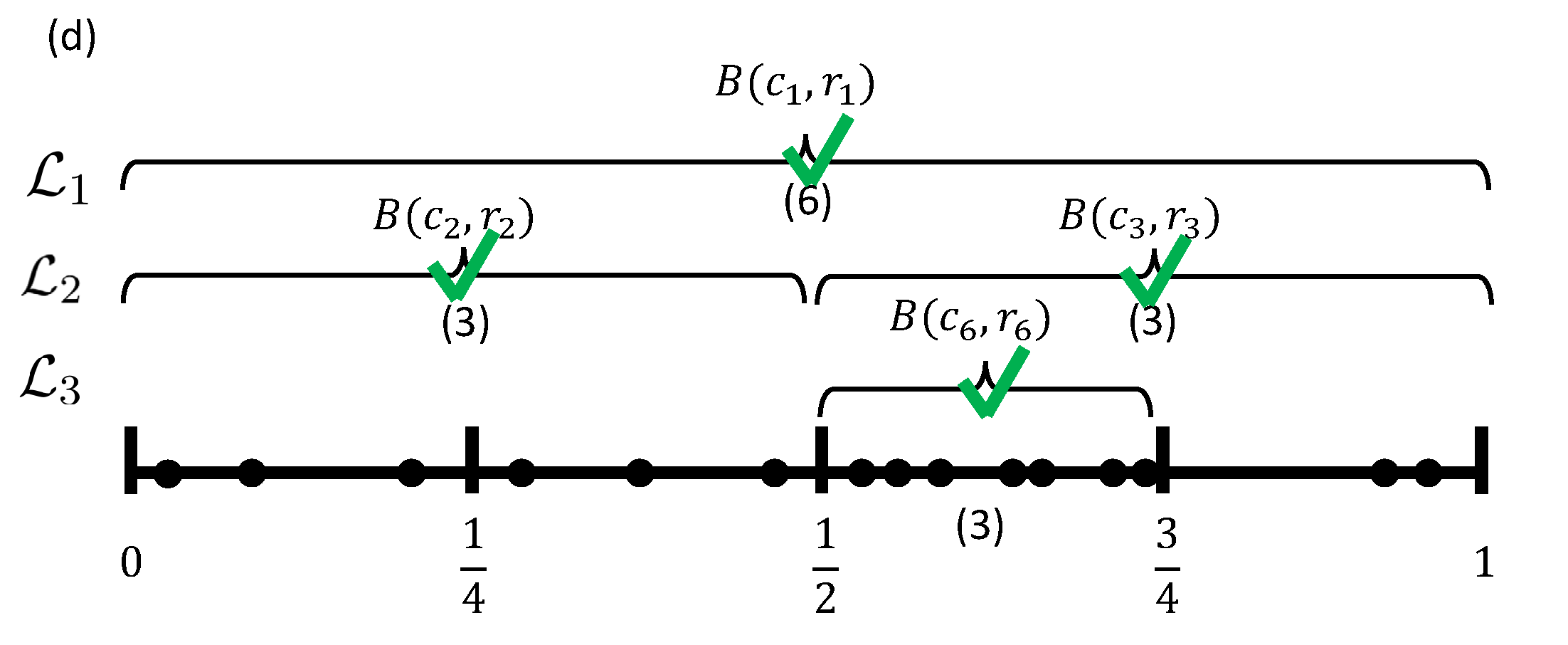}\\
\includegraphics[width=8.5cm]{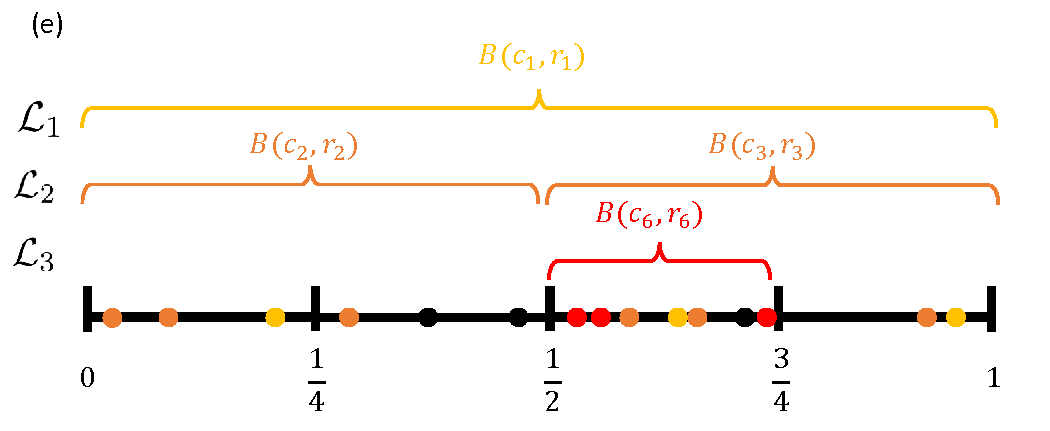}%
\end{minipage}
\hfill{}%
\begin{minipage}[c]{0.48\columnwidth}%
\footnotesize
\begin{itemize}
\vspace{0.35in}
\item[] The initial complete RP scheme as a binary tree with 3-layers on $[0,1]$.

\vspace{0.9in}
\item[] Starting from layer 3, we prune $B(\boldsymbol{c}_{7},\boldsymbol{w}_{7})$ since there are only $2<3$ observations available. We keep $B(\boldsymbol{c}_{6},\boldsymbol{w}_{6}),B(\boldsymbol{c}_{5},\boldsymbol{w}_{5}),B(\boldsymbol{c}_{4},\boldsymbol{w}_{4})$
as they all contain at least $m=3$ observations.

\vspace{0.25in}
\item[] Moving to layer 2, $B(\boldsymbol{c}_{3},\boldsymbol{w}_{3})$ has the 6 observations required by itself and its child $B(\boldsymbol{c}_{6}, \boldsymbol{w}_{6})$. Checking $B(\boldsymbol{c}_{2},\boldsymbol{w}_{2}),$ it contains only 6 observations so we prune its children $B(\boldsymbol{c}_{4},\boldsymbol{w}_{4}),B(\boldsymbol{c}_{5},\boldsymbol{w}_{5}).$

\vspace{0.35in}
\item[] Moving to layer 1, $B(\boldsymbol{c}_{1},\boldsymbol{w}_{1})$ contains
$15$ observations, therefore there are sufficient observations for $B(\boldsymbol{c}_{1},\boldsymbol{w}_{1})$
and its children $B(\boldsymbol{c}_{2},\boldsymbol{w}_{2})$ , $B(\boldsymbol{c}_{3},\boldsymbol{w}_{3})$ and  $B(\boldsymbol{c}_{6},\boldsymbol{w}_{6})$. We keep $B(\boldsymbol{c}_{1},\boldsymbol{w}_{1})$
and its children, completing the partitioning.

\vspace{0.4in}
\item[] Given the final RP scheme  $B(\boldsymbol{c}_{1},\boldsymbol{w}_{1})$, $B(\boldsymbol{c}_{2},\boldsymbol{w}_{2})$,
$B(\boldsymbol{c}_{3},\boldsymbol{w}_{3})$ and $B(\boldsymbol{c}_{6},\boldsymbol{w}_{6}),$ one possible random selection of the pseudo-inputs $\bar{\mathcal{X}_{j}}$ for the $j=1,\ldots,N$
different additive components (here $N=4$) conditional on the RP scheme is shown as colored dots.  Points with the same
color belong to blocks on the same layer.
\end{itemize}
\end{minipage}
\caption{\label{fig:Binary tree on 0_1}RP scheme on the domain $\mathsf{X}=[0,1]$
as a $2^{1}$-ary tree with 3 layers and $m=3$ pseudo-inputs per block. The $n=15$ data points $\mathcal{X}$ are represented as dots. The right pane describes the application of the RP pruning Algorithm 1 (a)-(d), and the selection of pseudo-inputs given the RP scheme in (e).  The left pane provides the analogous graphical construction of the RP scheme.
}
\end{figure}
To facilitate manipulating and storing the RP scheme on computer, we encode each block by its
centroid $\boldsymbol{c}_{j}=(c_{j}^{1},\cdots,c_{j}^{d})$ and half-width
$\boldsymbol{w}_{j}=(w_j^1,\ldots,w_j^d)$ where for simplicity we take the half-widths to be the same in each dimension given a layer $l$, $w_j^k=R_{l}, k=1,\ldots,d.$ The $j$-th block is then defined as 
\[
B_{j}\coloneqq B(\boldsymbol{c}_{j},\boldsymbol{w}_{j})=\left\{ \left.\boldsymbol{x}=(x^{1},\ldots,x^{d})\in\mathbb{R}^{d}\right||x^{k}-c_{j}^{k}|\leq w_{j}^k, k=1,\ldots,d\right\} .
\]
We require each block to have a minimum of $m_{j}$ observations,
allowing us to later define an SGP with $m_{j}$ pseudo-inputs
in each $B_{j}$. For simplicity of exposition, we will assume $m_{j}=m$ for all $j=1,\ldots,N.$ We also require pseudo-inputs $\bar{\mathcal{X}}^{(j)}$ to be mutually disjoint, so that each input setting in $\mathsf{X}$ is chosen as
a pseudo-input at most once.

An example RP scheme construction with $L=3$ layers and $m=3$ pseudo-inputs is shown in Figure \ref{fig:Binary tree on 0_1}.
The construction starts 
with a complete $2^d$-ary tree consisting of layers  $\mathcal{L}_{1}=\{B_{1}\},\mathcal{L}_{2}=\{B_{2},B_{3}\}$
and $\mathcal{L}_{3}=\{B_{4},B_{5},B_{6},B_{7}\},$ and a dataset of $n=15$ observations, shown as black dots.
Then, the complete tree is pruned according to Algorithm \ref{alg:pruning-algorithm-1}, which ensures
that each block $B_j$ will have at least $m$ pseudo-inputs available while also satisfying the
required properties P1 and P2.
Finally, given an RP scheme, one possible random selection of pseudo-inputs is shown.

Algorithm \ref{alg:pruning-algorithm-1} is able to perform the required pruning in general. 
Essentially, the algorithm works by requiring that the total number of observations
in block $B_j$ and all of $B_j$'s children satisfies
the total required number of pseudo-inputs for these components. Starting
from the bottom layer, Algorithm \ref{alg:pruning-algorithm-1} recursively
works up the tree, pruning sub-trees that do not satisfy this constraint
on total number of observations. Once the pruning is complete, the random selection of pseudo-inputs to blocks
can be drawn by starting with blocks in layer $L$ and working back to $B_{1},$ thereby
guaranteeing that all blocks meet the minimum of $m$ pseudo-inputs per block.
\begin{algorithm}[t]
\caption{\label{alg:pruning-algorithm-1} Pruning algorithm for RP scheme.}

\SetKwData{This}{this} \SetKwInOut{Input}{Input}\SetKwInOut{Output}{Output}
\Input{ RP partition scheme $\mathcal{A}$ consisting of $N$ components,
\; Observed dataset $\{\mathcal{X},\boldsymbol{y}\}$. \; } \Output{
RP partition scheme $\mathcal{A}^{'}$ consisting of $N^{'}\leq N$
components, \; } \For{$l$ \emph{in} $L:1$}{ \For{each component
$j$ \emph{in} the $l$-th layer $\mathcal{L}_{l}$}{ 



\For{$s$ \emph{in} $L:l$}{ 

$m_{req}$ $\leftarrow$ Sum of the numbers of pesudo-inputs required
for \emph{all} components contained in $B(\boldsymbol{c}_{j},\boldsymbol{w}_{j})$
in $\mathcal{A}^{'}$.\;

\uIf{ $\mid\mathcal{X}\cap B(\boldsymbol{c}_{j},\boldsymbol{w}_{j})\mid\geq m_{req}$
} { break\;} \Else{ Remove \emph{all} the children components
of component $j$ from the model in $s$-th layer.\; }


} } } 
\end{algorithm}

\subsection{SAGP Model}

Given a (pruned) RP scheme $\mathcal{B}_N$, we propose the additive model (\ref{eq:SAGP model}) for the response
$\boldsymbol{y}$, where each component $f_j=\left(f_{j}(\boldsymbol{x}_{1}),\ldots,f_{j}(\boldsymbol{x}_{n})\right)^{T}$ is described by an SGP model on the domain
$B_{j}$ and $\boldsymbol{\epsilon}\sim N_{n}(\boldsymbol{0}_{n},\sigma_{\epsilon}^{2}\boldsymbol{I}_{n}).$
For block $B(\boldsymbol{c}_{j},\boldsymbol{w}_{j})$, we use $\bar{\mathcal{X}}^{(j)}$ to
denote the pseudo-inputs for that block,
\begin{align*}
\bar{\mathcal{X}}^{(j)} & =\{\bar{\boldsymbol{x}}_{1}^{(j)},\ldots,\bar{\boldsymbol{x}}_{m_{j}}^{(j)}\}\subset B(\boldsymbol{c}_{j},\boldsymbol{w}_{j}) \text{ s.t. } \bar{\boldsymbol{x}}_{k}^{(j)}\in\mathsf{X}\, \forall k.
\end{align*}
Then, the SGP associated with $\boldsymbol{f}_{j}$ and pseudo-inputs $\bar{\mathcal{X}}^{(j)}$ has corresponding pseudo-targets, 
\begin{alignat}{1}
\bar{\boldsymbol{f}}_{j} & =\left(f_{j}(\bar{\boldsymbol{x}}_{1}^{(j)}),f_{j}(\bar{\boldsymbol{x}}_{2}^{(j)}),\ldots,f_{j}(\bar{\boldsymbol{x}}_{m_{j}}^{(j)})\right)^{T}\in\mathbb{R}^{m_{j}},j=1,\ldots,N.
\end{alignat}

Then, conditional on the RP scheme and pseudo-inputs,
the joint posterior of pseudo-targets and other parameters in (\ref{eq:SAGP model})
can be written as: 
\begin{align}
 & P(\bar{\boldsymbol{f}}_{1},\ldots,\bar{\boldsymbol{f}}_{N},\boldsymbol{\boldsymbol{\kappa}},\boldsymbol{\sigma}_{\epsilon}^{2}\mid\mathcal{B}_{N},\boldsymbol{y},\mathcal{X},\bar{\mathcal{X}}^{(1)},\ldots,\bar{\mathcal{X}}^{(N)})\propto\nonumber \\
 & \underbrace{P(\boldsymbol{y}\mid\bar{\boldsymbol{f}}_{1},\ldots,\bar{\boldsymbol{f}}_{N},\boldsymbol{\kappa},\boldsymbol{\sigma}_{\epsilon}^{2},\mathcal{B}_{N},\bar{\mathcal{X}}^{(1)},\ldots,\bar{\mathcal{X}}^{(N)},\mathcal{X})}_{\text{Likelihood Function}} \underbrace{P(\bar{\boldsymbol{f}}_{1},\ldots,\bar{\boldsymbol{f}}_{N}\mid\mathcal{B}_{N},\bar{\mathcal{X}}^{(1)},\ldots,\bar{\mathcal{X}}^{(N)},\mathcal{X},\boldsymbol{\boldsymbol{\kappa}})}_{\text{Pseudo-target Prior}}\times\nonumber \\
 & \underbrace{P(\boldsymbol{\boldsymbol{\kappa}}\mid\mathcal{B}_N)}_{\text{Kernel Prior}}\underbrace{P(\boldsymbol{\sigma}_{\epsilon}^{2})}_{\text{Error Prior}}.\label{eq:full likelihood}
\end{align}

In effect, we view the choice of pseudo-inputs $\bar{\mathcal{X}}^{(1)},\ldots,\bar{\mathcal{X}}^{(N)}$
as nuisance parameters, and ultimately will integrate them out with
respect to the prior $P(\bar{\mathcal{X}}^{(1)},\ldots,\bar{\mathcal{X}}^{(N)}\mid\mathcal{B}_{N})$,
which gives the marginal posterior of interest,
\begin{align*}
P(\bar{\boldsymbol{f}}_{1},\ldots,\bar{\boldsymbol{f}}_{N},\boldsymbol{\boldsymbol{\kappa}},\boldsymbol{\sigma}_{\epsilon}^{2}\mid\mathcal{B}_{N},\boldsymbol{y},\mathcal{X}) & = \\ 
\int P(\bar{\boldsymbol{f}}_{1},\ldots,\bar{\boldsymbol{f}}_{N},\boldsymbol{\boldsymbol{\kappa}},\boldsymbol{\sigma}_{\epsilon}^{2}\mid\mathcal{B}_{N},\boldsymbol{y},\mathcal{X},\bar{\mathcal{X}}^{(1)},\ldots,\bar{\mathcal{X}}^{(N)})
 & \underbrace{P(\bar{\mathcal{X}}^{(1)},\ldots,\bar{\mathcal{X}}^{(N)}\mid\mathcal{B}_{N})}_{\text{Pseudo-input Prior}} d\bar{\boldsymbol{x}}^{m_{1}}\ldots d\bar{\boldsymbol{x}}^{m_{N}},
\end{align*}
where $d\bar{\boldsymbol{x}}^{m_{j}}=d\bar{\boldsymbol{x}}_{1}^{(j)}\times\ldots\times d\bar{\boldsymbol{x}}_{m_{j}}^{(j)}.$

In subsection \ref{sec:Sampling Algorithm}
we will show a Gibbs sampler algorithm for SAGP fitting  and for calculating predictions, but first we describe in greater detail the likelihood function and various prior distributions involved in the SAGP model.


\subsubsection{\label{sec:Specification of Model Likelihoods}\label{sec:Conditional Likelihood of FULL OBSERVATIONS_FULL INPUTS}
Likelihood Function, \texorpdfstring{$P({\boldmath y}\mid\bar{\boldmath{f}}_{1},\ldots,\bar{\boldmath{f}}_{N}, \boldsymbol{\kappa},\sigma_{\epsilon}^{2},\mathcal{B}_{N},\bar{\mathcal{X}}^{(1)},\ldots,\bar{\mathcal{X}}^{(N)},\mathcal{X})$}{P}}

Let us denote the covariance
kernel for the SGP in the $j$-th block by $K^{(j)},j=1,\ldots,N$ . We use
the Gaussian covariance kernel supported inside $B_{j}$
with parameters $\boldsymbol{\kappa}^{(j)}=(\rho^{(j)},\eta^{(j)})$.  
We have
\begin{align}
K^{(j)}(\boldsymbol{x},\boldsymbol{x}') & \coloneqq\frac{1}{\eta^{(j)}}\cdot\left(\rho^{(j)}\right)^{\left[(\boldsymbol{x}-\boldsymbol{x}')^{T}(\boldsymbol{x}-\boldsymbol{x}')\right]},\forall\boldsymbol{x},\boldsymbol{x}'\in B_{j}.\label{eq:kernel function}
\end{align}

Using (\ref{eq:kernel function}),  
we can write down the (cross-)covariance matrices among and between inputs in $\mathcal{X}$ and $\bar{\mathcal{X}}^{(j)}$
as: 
\begin{align*}
\boldsymbol{K}_{n}^{(j)} & \coloneqq\left[K^{(j)}(\boldsymbol{x}_{k},\boldsymbol{x}_{l})\right]_{k,l=1}^{n},\\
\boldsymbol{K}_{m_{j}}^{(j)} & \coloneqq\left[K^{(j)}(\bar{\boldsymbol{x}}_{k}^{(j)},\bar{\boldsymbol{x}}_{l}^{(j)})\right]_{k,l=1}^{m_{j}},\\
\boldsymbol{K}_{nm_{j}}^{(j)} & \coloneqq\left[K^{(j)}(\boldsymbol{x}_{k},\bar{\boldsymbol{x}}_{l}^{(j)})\right]_{k,l=1}^{n,m_{j}}=\left(\boldsymbol{K}_{m_{j}n}^{(j)}\right)^{T}.
\end{align*}
For a general $\boldsymbol{x}\in\mathbb{R}^{d},$ we also have 
\begin{align}
\boldsymbol{k}_{\boldsymbol{x}}^{(j)} & \coloneqq\left(\begin{array}{ccc}
K^{(j)}(\bar{\boldsymbol{x}}_{1}^{(j)},\boldsymbol{x}), & \ldots & ,K^{(j)}(\bar{\boldsymbol{x}}_{m_{j}}^{(j)},\boldsymbol{x})\end{array}\right)^{T}.\label{eq:single vector shot}
\end{align}

Assuming the additive components are conditionally independent, the likelihood is
(see Lemma
\ref{lem:Conditional GP closed form Lemma} in Appendix \ref{sec:Detailed Derivation of Posterior Distribution})
\begin{align}
P(\boldsymbol{y}\mid\boldsymbol{\bar{f}}_{1},\ldots,\bar{\boldsymbol{f}}_{N},\boldsymbol{\boldsymbol{\kappa}},\boldsymbol{\sigma}_{\epsilon}^{2},\mathcal{B}_{N},\bar{\mathcal{X}}^{(1)},\ldots,\bar{\mathcal{X}}^{(N)}) & =N_{n}\left(\boldsymbol{y}\left|\sum_{j=1}^{N}\boldsymbol{K}_{nm_{j}}^{(j)}\left(\boldsymbol{K}_{m_{j}}^{(j)}\right)^{-1}\bar{\boldsymbol{f}}_{j},\sigma_{\epsilon}^{2}\boldsymbol{I}_{n}+\sum_{j=1}^{N}\boldsymbol{\Lambda}_{n}^{(j)}\right.\right)\label{eq:multiple obs pred}
\end{align}
where the matrix $\boldsymbol{\Lambda}_{n}^{(j)}\coloneqq\text{diag}\left(K_{ii}^{(j)}-\boldsymbol{k}_{i}^{(j)T}\left(\boldsymbol{K}_{m_{j}}^{(j)}\right)^{-1}\boldsymbol{k}_{i}^{(j)}\right)_{n\times n}$
takes the diagonal form, with $\boldsymbol{k}_{i}^{(j)}$ as defined
in (\ref{eq:single vector shot}) (with subscript $i$ being shorthand for $\boldsymbol{x}_i$).

\subsubsection{\label{sec:Conditional Likelihood of PTs}Pseudo-target Prior,\texorpdfstring{ $P(\bar{\boldmath{f}}_{1},\ldots,\bar{\boldmath{f}}_{N}\mid\mathcal{B}_{N},\bar{\mathcal{X}}^{(1)},\ldots,\bar{\mathcal{X}}^{(N)},\mathcal{X},\boldmath{\kappa})$}{Pfbar}}

The prior distribution of pseudo-targets given pseudo-inputs
and covariance function parameters is straight-forward. Following \citet{Snelson&Ghahramani2006},
the pseudo-targets 
are assumed to be a priori conditionally independent, and so
 have Gaussian distributions with prescribed kernels: 
\begin{align}
P(\bar{\boldsymbol{f}}_{1},\ldots,\bar{\boldsymbol{f}}_{N}\mid\mathcal{B}_{N},\bar{\mathcal{X}}^{(1)},\ldots,\bar{\mathcal{X}}^{(N)},\mathcal{X},\boldsymbol{\boldsymbol{\kappa}}) & =\prod_{j=1}^{N}P(\bar{\boldsymbol{f}}_{j}\mid\mathcal{B}_{N},\bar{\mathcal{X}}^{(j)},\mathcal{X},\boldsymbol{\boldsymbol{\kappa}}^{(j)})\nonumber \\
 & =\prod_{j=1}^{N}N_{m_{j}}\left(\bar{\boldsymbol{f}}_{j}\left|\boldsymbol{0}_{m_{j}},\boldsymbol{K}_{m_{j}}^{(j)}\right.\right).\label{eq: joint Pseuo-target likelihood}
\end{align}

\subsubsection{\label{sec:Conditional likelihood of PIs}Pseudo-input Prior,\texorpdfstring{ $P(\bar{\mathcal{X}}^{(1)},\ldots,\bar{\mathcal{X}}^{(N)}\mid\mathcal{B}_{N},\mathcal{X})$}{PXbar}}

The idea of the proposed pseudo-input prior is to sample pseudo-inputs uniformly within each block $B_j$ while satisfying properties P1--P3 required for the RP scheme, $\mathcal{B}_N.$  Algorithm \ref{alg:PI sampling} implements such a sampling scheme, which we now motivate.
Let the index set $\mathcal{I}_{j}$
representing the indices of  children blocks of block $B_j$, which is defined as 
\begin{align*}
\mathcal{I}_{j} & \coloneqq\{k\neq j\text{ such that }B_k\subset B_j\},
\end{align*}
and also define the collection of already selected pseudo-inputs of these child blocks
as $\mathcal{C}(B_j)\coloneqq
\cup_{k\in\mathcal{I}_{j}}\bar{\mathcal{X}}^{(k)}.$
Then, 
\begin{align*}
P(\bar{\mathcal{X}}^{(1)},\ldots,\bar{\mathcal{X}}^{(N)}\mid\mathcal{B}_{N}) & =\prod_{\ell=L}^{1}\prod_{j: B_j\in\mathcal{L}_{\ell}}P(\bar{\mathcal{X}}^{(j)}\mid\mathcal{C}(B_j))
\end{align*}
where
\begin{align*}
P(\bar{\mathcal{X}}^{(j)}\mid\mathcal{C}(B_j)) & =\prod_{i=1}^{m_{j}}P(\bar{\boldsymbol{x}}_{i}^{(j)}\mid\mathcal{C}(B_j))
\end{align*}
and
\begin{align*}
P(\bar{\boldsymbol{x}}_{i}^{(j)}\mid\mathcal{C}(B_j)) & =\text{Discrete Uniform}\left(\left\{ \boldsymbol{x}\in\mathcal{X}\subset B_j\backslash\mathcal{C}(B_j)\right\} \right).
\end{align*}
In the expression above, we essentially draw a random sample from
all those observed locations that have not been selected as pseudo-inputs
of any children components in the lower layers of the RP  scheme. 

Unlike the standard SGP approach, this allows us to capture the uncertainty of pseudo-input selection by sampling the pseudo-inputs using Algorithm \ref{alg:PI sampling}
and propagating this uncertainty to the posterior.
Alternatives such as a
continuous uniform prior over each component domain $B_j$,
or sampling accordingly to design-theoretic considerations \citep{pratola2019optimal},
are possible.  

\begin{algorithm}[t]
\caption{\label{alg:PI sampling}Sampling  
 pseudo-inputs given RP scheme $\mathcal{B}_N$.}

\setcounter{AlgoLine}{0} \SetKwData{This}{this} \SetKwInOut{Input}{Input}\SetKwInOut{Output}{Output}
\Input{ RP scheme $\mathcal{B}_N$ consisting of $N$ blocks, \; Observed
inputs $\mathcal{X}$. \; } \Output{
Sample of $\bar{\mathcal{X}}^{(j)},j=1,\ldots,N$ conditional on RP scheme $\mathcal{B}_N$. } 

Initialize $\mathcal{X}_{A}=\mathcal{X}$ as available inputs. 

\For{$j$ in $N:1$}{




Sample a random sample $\bar{\mathcal{X}}^{(j)}\subset\mathcal{X}\subset\mathbb{R}^{d}$
of size $m_{j}$ from $\mathcal{X}\cap B_j\cap\mathcal{X}_{A}$.\;

$\mathcal{X}_{A} \leftarrow \mathcal{X}_{A}\setminus\bar{\mathcal{X}}^{(j)}$ \texttt{// Remove $\bar{\mathcal{X}}^{(j)}$ sampled in the
previous step from $\mathcal{X}_{A}$.}\; 
} 
\end{algorithm}

\subsubsection{\label{sec:Priors and Hyper-parameters}Additional Prior Distributions, \texorpdfstring{$P(\boldsymbol{\kappa}\mid\mathcal{B}_N)$ and  $P(\sigma_{\epsilon}^{2})$}{Pkappa}}

We place a conjugate inverse gamma prior on the noise variance, $\sigma_{\epsilon}^{2}$,
\begin{align*}
\sigma_{\epsilon}^{2} & \sim\text{InverseGamma}(\alpha_{\epsilon},\beta_{\epsilon}).
\end{align*}
The hyper-parameters $\alpha_{\epsilon}$ and $\beta_{\epsilon}$ may be chosen as the hyper-parameters of the noise variance in traditional Bayesian GP regression. 

We assume independent priors on the scale and correlation parameters of the kernel, $\eta^{(j)}$ and $\rho^{(j)}$, 
\begin{align*}
P(\boldsymbol{\boldsymbol{\kappa}}) & =P(\boldsymbol{\boldsymbol{\kappa}}^{(1)},\ldots,\boldsymbol{\boldsymbol{\kappa}}^{(j)})
=\prod_{\ell=1}^{L}\prod_{B_j\in\mathcal{L}_{\ell}}P(\eta^{(j)}\mid\alpha_{\eta}^{l},\beta_{\eta}^{l}) P(\rho^{(j)}).
\end{align*}

The precision parameters $\eta^{(j)}$ are assumed to have gamma priors,
\begin{align*}
\eta^{(j)} & \sim\text{Gamma}(\alpha_{\eta}^{l},\beta_{\eta}^{l}),
\end{align*}
with $\alpha_{\eta}^{l},\beta_{\eta}^{l}>0,l=1,\ldots,L$. The hyper-parameters $\alpha_{\eta}^{l},\beta_{\eta}^{l}$ are the same for components within the same layer. We set up these hyper-parameters so that the variance of the response explained by the SAGP model is unequally partitioned across the $L$ layers, with components in higher layers of the partitioning scheme explaining larger portions of the variance. To facilitate the set-up of the hyper-parameters, we first normalize the observed responses $y_{1},\ldots,y_{n}$, re-centering and re-scaling so that they have mean 0 and variance 1. 
For all the components in layer $l$, we set 
\begin{align*}
\alpha_{\eta}^{l} = & \mathsf{c}_{1\eta} + 1, \\
\beta_{\eta}^{l} = & \mathsf{c}_{1\eta} (1-\mathsf{c}_{2\eta}) \mathsf{c}_{2\eta} ^{l-1},
\end{align*}
with $\mathsf{c}_{1\eta}>0$ and $0<\mathsf{c}_{2\eta}<1$. For each component $j$ in layer $l$, this choice implies that $1/\eta^{(j)}$, the marginal variance of the component, has prior mean
\begin{align*}
E[1/\eta^{(j)}] = & \frac{\beta_{\eta}^{l}}{\alpha_{\eta}^{l} - 1} = (1-\mathsf{c}_{2\eta}) \mathsf{c}_{2\eta}^{l-1}.
\end{align*}
For example, if $\mathsf{c}_{2\eta}=.1$, components on layer $l=1$ are expected to have variance $1-\mathsf{c}_{2\eta} = .9$, which is 90$\%$ of the variance of the response because of the normalization. Components on layer $l=2$ are expected to have  $(1-\mathsf{c}_{2\eta}) \mathsf{c}_{2\eta} = 0.09$, 9$\%$ of the variance of the response. Similarly, as $l$ increases, components are expected to explain smaller portions of the variance of the response. In particular, the geometric decay of the prior mean of $1/\eta^{(j)}$ is chosen so that the expected layer-specific variances add up to approximately the total response variance, which is guaranteed because, if $L$ is sufficiently big, $\sum_{l=1}^L  (1-\mathsf{c}_{2\eta}) \mathsf{c}_{2\eta}^{l-1} \approx 1$. The other hyper-parameter $\mathsf{c}_{1\eta}$ controls the spread of the prior distributions of $1/\eta^{(j)}$, with larger values of $\mathsf{c}_{1\eta}$ imposing a tighter constraint to the prior mean. In our experience, values of $\mathsf{c}_{2\eta} = .1$ and $\mathsf{c}_{1\eta}$ between 10 and 50 appear to provide the best results in our applications.

We set the prior distributions on the parameters $\rho^{(j)}$ in the following way. First of all, we assume that the inputs $\boldsymbol{x}_i$'s have been appropriately scaled, so that the domain $\mathsf{X}$ is mapped into the unit cube $[0,1]^d$. This facilitates the definition of priors for $\rho^{(j)}$. Second, as for the $\eta^{(j)}$, we assume the same prior distribution for parameters corresponding to components in the same layer $l$. Third, we adopt a structure of priors imposing smoother behavior for components in the top layers of the partitioning scheme. In other words, we impose a structure of priors where $\rho^{(j)}$ is expected to be greater than $\rho^{(j')}$ if component $j$ belongs to a layer on a higher level than the layer of component $j'$. Despite a family of beta priors on the $\rho^{(j)}$ may be tuned to satisfy these properties, we empirically observed that setting the values of these parameters to fixed layer-specific constants $\rho_{l}$ (i.e., $P(\rho^{(j)}) = \delta_{\rho_{l}}$, the Dirac delta function on $\rho_{l}$) worked as well but was computationally less expensive. To get the sense on how the values of $\rho_{l}$ affect the layer-specific correlations, one may plot the correlation for two responses as a function of the distance of their inputs, as specified by Equation (\ref{eq:kernel function}). We provide this plot in Appendix \ref{sec:plot_correlation}, in the case $L=5$ and using the values $\rho_{1}=10^{-1}$, $\rho_{L}=10^{-50}$ and the
intermediate values $\rho_{l},l=2,\ldots,L-1$ to be equally spaced
between $\rho_{1}$ and $\rho_{L}$ on the logarithm (base 10) scale. Even though these values of $\rho_{l}$ may appear to quickly become excessively small, the sizes of the subdomains where the components are defined (i.e., the blocks $B_j$) shrink as $l$ increases. In our numerical example, if we consider a one-dimensional case with two inputs $\boldsymbol{x}$ and $\boldsymbol{x}'$ at distance 0.0625 (i.e., the largest distance between two points in one block on the fifth layer), the assumed correlations on components on layer $l = 1$ to 5 are 0.99, 0.89, 0.80, 0.71 and 0.64, respectively. Notably, the decay of such values depends on the number of layers $L$, which can be tuned using prior beliefs and the information in the data. In our applications, trading the conventional estimation of the parameters $\rho^{(j)}$ with a set of fixed $\rho_{l}$ and a data-driven selection of $L$ via cross-validation (see Section \ref{sec:Specification of Priors and Hyper-parameters}) resulted in sufficiently flexible models.

\subsubsection{\label{sec:Posterior for PTs}Full Conditional Distribution of Pseudo-targets }

In order to implement an MCMC algorithm for SAGP, we apply Bayes theorem on the pseudo-inputs $\bar{\boldsymbol{f}}_{j}$ in order
to yield its full conditional distribution from (\ref{eq:multiple obs pred})
and (\ref{eq: joint Pseuo-target likelihood}) and the conditional
independence assumption, 
\begin{align}
 & P(\bar{\boldsymbol{f}}_{j}\mid\boldsymbol{y},\mathcal{X},\bar{\boldsymbol{f}}_{1},\ldots,\bar{\boldsymbol{f}}_{j-1},\bar{\boldsymbol{f}}_{j+1},\ldots\bar{\boldsymbol{f}}_{N},\mathcal{B}_{N},\bar{\mathcal{X}}^{(1)},\ldots,\bar{\mathcal{X}}^{(N)},\boldsymbol{\boldsymbol{\kappa}},\boldsymbol{\sigma}_{\epsilon}^{2})\nonumber \\
 & \propto P(\boldsymbol{y}\mid\bar{\boldsymbol{f}}_{1},\ldots\bar{\boldsymbol{f}}_{N},\mathcal{B}_{N},\bar{\mathcal{X}}^{(1)},\ldots,\bar{\mathcal{X}}^{(N)},\boldsymbol{\boldsymbol{\kappa}},\boldsymbol{\sigma}_{\epsilon}^{2})\times\nonumber \\
 & P(\bar{\boldsymbol{f}}_{j}\mid\mathcal{X},\mathcal{B}_{N},\bar{\mathcal{X}}^{(j)},\boldsymbol{\boldsymbol{\kappa}},\boldsymbol{\sigma}_{\epsilon}^{2})\nonumber \\
 & =N_{n}\left(\left.\boldsymbol{r}_{j}\right|\boldsymbol{K}_{nm_{j}}^{(j)}\left(\boldsymbol{K}_{m_{j}}^{(j)}\right)^{-1}\bar{\boldsymbol{f}}_{j},\boldsymbol{\Lambda}_{n}^{(j)}+\sigma_{\epsilon}^{2}\boldsymbol{I}_{n}\right)\times
 N_{m_{j}}\left(\left.\bar{\boldsymbol{f}}_{j}\right|\boldsymbol{0}_{m_{j}},\boldsymbol{K}_{m_{j}}^{(j)}\right),\label{eq:Posterior of pseudo-targets}
 \end{align}
where $\boldsymbol{r}_{j}=\boldsymbol{y}-\sum_{l\neq j}\boldsymbol{K}_{nm_{l}}^{(l)}\left(\boldsymbol{K}_{m_{l}}^{(l)}\right)^{-1}\bar{\boldsymbol{f}}_{l}.$
Using normal-normal conjugacy, we can identify the mean and variance
of this normal distribution, $\bar{\boldsymbol{f}}_{j}\mid\boldsymbol{Mean}_{j},\boldsymbol{Var}_{j}$,
where (see Appendix \ref{sec:Detailed Derivation of Posterior Distribution})

\begin{align}
\boldsymbol{Mean}_{j} & =\boldsymbol{K}_{m_{j}}^{(j)}\boldsymbol{Q}_{m_{j}}^{(j)-1}\boldsymbol{K}_{m_{j}n}^{(j)}\left(\boldsymbol{\Lambda}_{n}^{(j)}+\sigma_{\epsilon}^{2}\boldsymbol{I}_{n}\right)^{-1}\boldsymbol{r}_{j},\nonumber \\
\boldsymbol{Var}_{j} & =\boldsymbol{K}_{m_{j}}^{(j)}\boldsymbol{Q}_{m_{j}}^{(j)-1}\boldsymbol{K}_{m_{j}}^{(j)},\nonumber \\
\text{and }\boldsymbol{Q}_{m_{j}}^{(j)}= & \boldsymbol{K}_{m_{j}}^{(j)}+\boldsymbol{K}_{m_{j}n}^{(j)}\left(\boldsymbol{\Lambda}_{n}^{(j)}+\sigma_{\epsilon}^{2}\boldsymbol{I}_{n}\right)^{-1}\boldsymbol{K}_{nm_{j}}^{(j)}.\label{eq:PT posterior closed form}
\end{align}
Although we still need to invert an $n\times n$ matrix $\boldsymbol{\Lambda}_{n}^{(j)}+\sigma_{\epsilon}^{2}\boldsymbol{I}_{n}$,
it is a diagonal matrix and hence its computational cost will be $\mathcal{O}(n)$.

\subsubsection{\label{subsec:Full-Conditional-Distribution}Full Conditional Distribution
of Noise Variance}

As we mentioned in the previous section, we want to make use of the
Gaussian-inverse gamma conjugacy. For the observation of sample size
$n$, by conjugacy, the distribution $\sigma_{\epsilon}^{2}$ is again
inverse gamma,
\begin{align*}
& P(\sigma_{\epsilon}^{2}  \mid\boldsymbol{y},\mathcal{X},\bar{\boldsymbol{f}}_{1},\ldots\bar{\boldsymbol{f}}_{N},\mathcal{B}_{N},\bar{\mathcal{X}}^{(1)},\ldots,\bar{\mathcal{X}}^{(N)},\boldsymbol{\boldsymbol{\kappa}}) =\\
& \text{InverseGamma}\left(\alpha_{\epsilon}+\frac{n}{2},\beta_{\epsilon}+\frac{1}{2}(\boldsymbol{y}-\hat{\boldsymbol{y}})^{T}(\boldsymbol{y}-\hat{\boldsymbol{y}})\right),
\end{align*}
where $\hat{\boldsymbol{y}}\coloneqq\sum_{j=1}^{N}\boldsymbol{K}_{nm_{j}}^{(j)}\left(\boldsymbol{K}_{m_{j}}^{(j)}\right)^{-1}\bar{\boldsymbol{f}}_{j}$
is the ``fitted value'' from the SAGP model. We can directly sample
this parameter using a Gibbs step.

\subsubsection{\label{sec:Sampling Algorithm}Sampling Algorithm}

SAGP is fitted by a Metropolis-Hastings Markov chain Monte Carlo (MCMC)
algorithm \citep{Gelfand&Smith1990}. For each
additive component of the model, we have to use the partial residuals
$\boldsymbol{r}_{j}$ defined in (\ref{eq:Posterior of pseudo-targets})
as data. This step is from the back-fitting scheme designed for fitting
additive Bayesian models \citep{Hastie&Tibshirani2000}.

From the likelihood derivations presented in section \ref{sec:Posterior for PTs},
we know that $\bar{\boldsymbol{f}_{j}}$ can be directly sampled from
their conditional distributions for each $j=1,\ldots,N$ components.
The difficulty in this step is to compute the $\boldsymbol{Mean}_{j},\boldsymbol{Var}_{j}$
in (\ref{eq:PT posterior closed form}). As mentioned earlier, the
main computational cost occurs in the inversion of the covariance
matrices in section \ref{sec:Conditional Likelihood of FULL OBSERVATIONS_FULL INPUTS},
which has been reduced compared to a full GP covariance matrix. Numerical
instability in inversion of these matrices may cause additional problems,
so we adopt the Cholesky decomposition method with diagonal perturbation
to solve this instability problem as in \citet{Rasmuessen&Williams2006}.
For each $\eta_{j}$ we do not have normal conjugacy, therefore an
adaptive Metropolis-Hasting step is used for sampling $\eta_{j}$
\citep{Banerjee2012}.

The advantage of using such a fully Bayesian model is that the uncertainty
quantification comes naturally with the posterior samples from the
sampler. Our posterior inference below can be based on all these posterior
samples. The algorithm for overall sampling is presented in Algorithm
\ref{alg:MCMC-algorithm} in the Appendix~\ref{sec:MCMC-Algorithm-for}.

\subsubsection{\label{sec:Specification of Priors and Hyper-parameters}Tuning Parameters
and Complexity}

The trade-off between number of layers $L$ in the RP scheme and the
number of pseudo-inputs $m$ is central to the SAGP model.
On one hand, in SGP modeling \citep{Snelson&Ghahramani2006,Snelson&Ghahramani2007,Lee_etal2017},
we need to
increase the number of pseudo-inputs $m$ to get a better fit of the SGP model. On the other hand, for
regression tree partitionining models \citep{Chipman1998,Chipman2016,Pratola2017},
the more additive components a model has, the better fit we can expect.
In the SAGP model, increasing both factors (number of pseudo-inputs,
$m,$ and number of layers, $L,$)
would certainly improve the overall fit, but the interesting observation
is that there exists a trade-off between these two tuning parameters. Increasing the number of layers $L$ may counter-act the effect
of decreasing the number of pseudo-inputs $m$, and vice versa. Theoretically,
we can tune the choice of number of layers using cross-validation
(see Figure \ref{fig:CV_HR}). Practically, we can usually choose reasonable
$m$ and $L$ depending on the desired granularity of the RP scheme.

This trade-off between $m$ and $L$ can also be observed by considering
the model's computational complexity. We already mentioned above that for a full GP
model based on $\mathcal{X}$ the complexity is of order $\mathcal{O}(n^{3})$;
for an SGP model with $m\ll n$ pseudo-inputs selected from $\mathcal{X},$
the complexity is of order $\mathcal{O}(nm^{2})$ \citep{Snelson&Ghahramani2006}.
Since each additive component in our SAGP model is essentially
an SGP model,  the overall complexity is given by the following
proposition. \begin{prop} \label{prop:For-an-RP}For an RP scheme
on input-domain $[0,1]^{d}$, with $b_{i}$-ary tree \citep{storer2012introduction}
in the $i$-th dimension, the complexity of fitting an $L$-layer
SAGP model with $m$ pseudo-inputs for each block and an overall sample
of size $n$ is at most  $\mathcal{O}\left(\sum_{\ell=1}^{L}\prod_{i=1}^{d}b_{i}^{\ell-1}\cdot n\cdot m^{2}\right)$.
\end{prop}

\begin{proof} See Appendix \ref{sec:Proof-of-Proposition}. \end{proof}
For $N=1$, where there is only one layer and one component, this complexity reduces to SGP complexity
with $m$ pseudo-inputs. We will revisit this \emph{component number
pseudo-input trade-off} in our data analyses.

\section{\label{sec:Simulation Study}Simulation Study}

\subsection{Design}

To evaluate the performance of our methodology and compare it to competing
approaches, we run a family of simulations. We focus on the one-dimensional
case ($d=1$) and we simulate from a GP with mean function 
\begin{equation}
f(x)=-5-6x^{3}+30(x-.5)^{2}+3\exp(2x-1)+3x^{2}\sin(12\pi x)+\cos(6\pi x),\label{eq:simulated dta formula}
\end{equation}
which is represented in Figure \ref{fig:exSim} in the interval {[}0,1{]}.
We generate a sample of $n=200$ locations from a uniform distribution
on $[0,1]$ and we define the observed responses as $y_{i}=f(x_{i})+\epsilon_{i}$
for all $x_{i}\in\mathcal{X}$, with $\epsilon_{i}\sim N_{1}(0,0.1)$.
The data are split into training and testing sets, with sizes 150
and 50, respectively. We consider two scenarios. In the first scenario,
the testing set is selected at random. In the second scenario, the
testing set is chosen as the subset with 50 data points with $x_{i}$
closest to a point randomly chosen in $[0.25,0.75]$. Figure \ref{fig:exSim}
shows an example for each of these scenarios.

\begin{figure}[h!]
\centering \includegraphics[width=1\textwidth]{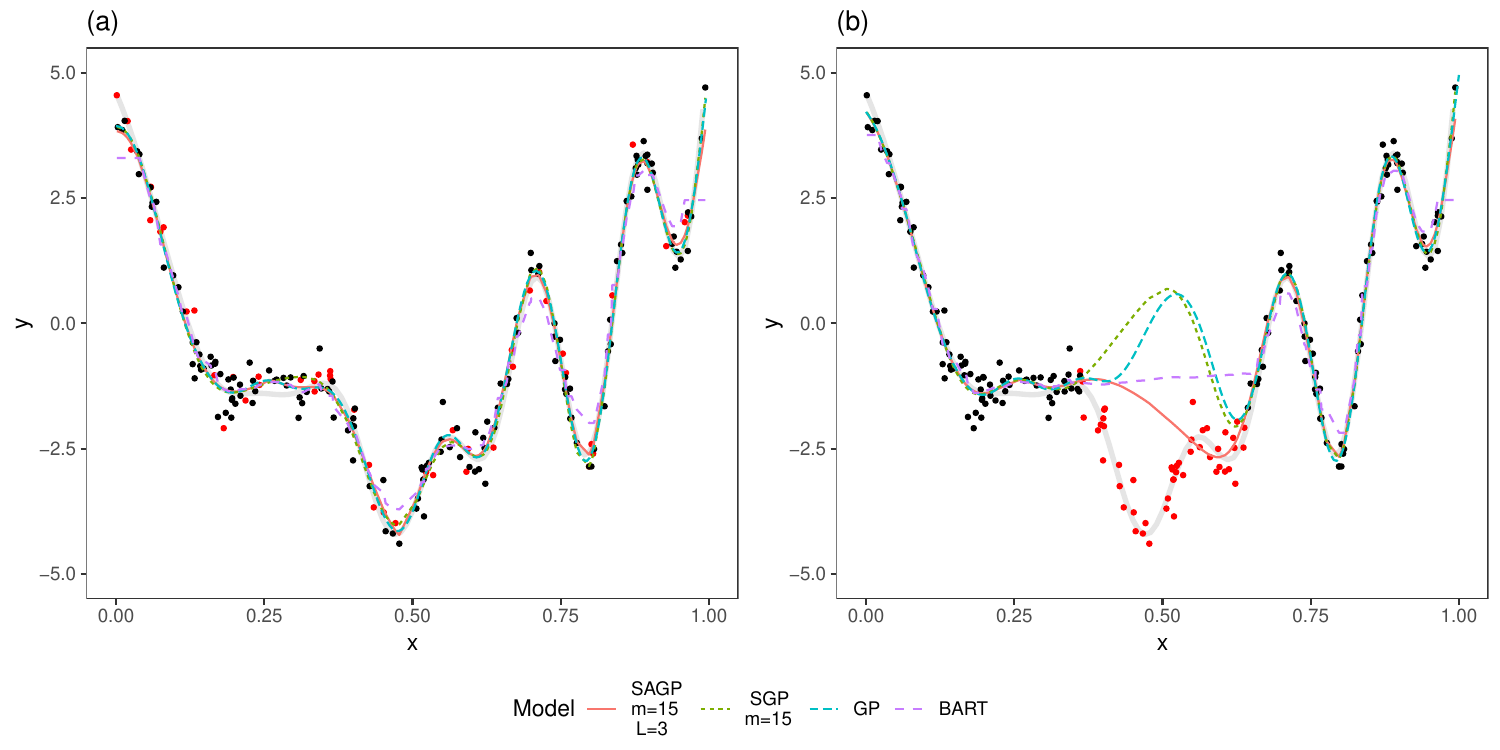}
\caption{Example of data generated in the simulation study. The gray, bold,
curve represents the true mean function $f(x)$. Training and testing
sets are represented as black and red points. Panels (a) and (b) show
the scenarios where the 50 data points of the testing set are chosen at random or as the input location that is closest to a randomly chosen point (0.5 in the example), respectively. The posterior
predictive functions of four models, fit on the training portion of
the data, are provided in both panels. \label{fig:exSim}}
\end{figure}

We generate 1000 datasets for each scenario (random and interval testing set). In each dataset, we fit the
SAGP model with three configurations of $m,L$: (i) $m=5$, $L=4$;
(ii) $m=10$, $L=3$; (iii) $m=15$, $L=3$. We compare the SAGP models
to the following methods: 
\begin{itemize}
\item Full GP regression. We use the implementation of GP regression model
in the R package \texttt{$\mathtt{DiceKriging}$} by \citet{roustant2012dicekriging}. 
\item SGP regression. We consider the choices $m=5$, $m=10$ and $m=15$
and use the implementation of SGP in the Matlab package implementation $\mathtt{SPGP}$ at  \url{http://www.gaussianprocess.org}
accompanying the paper by \citet{Snelson&Ghahramani2006}. 
\item Bayesian Additive Regression Trees (BART) \citet{Chipman1998,Chipman_etal2010,Pratola2017}.
We used the default number of trees as specified in \citet{Chipman_etal2010}
and the implementation at \url{http://bitbucket.org/mpratola/openbt}. 
\end{itemize}
For each generated dataset, the models are
fit on the training data and used to predict the response on the testing data. For each point in the testing set, we compute the
estimated mean function $\hat{y}(x_{i})$ (see Section \ref{subsec:Full-Conditional-Distribution})
and the 95$\%$ prediction interval (PI) for $y_{i}$. The performance
of the estimators of the mean function is evaluated in terms of root mean
squared error (RMSE). To assess the appropriateness
of the uncertainty quantification, we compute the coverage of the
PIs and compare it to the nominal prediction level. Finally, we compare
the methods in terms of average value of interval score, which is
a summary measure to assess the quality of prediction intervals \citep{gneiting2007strictly}.
Given a $(1-\alpha)100\%$ PI for $y_{i}$ with extremes $(l_{i},u_{i})$,
the interval score at $y_{i}$ is defined as

\[
s_{\alpha}(l_{i},u_{i};y_{i})=(u_{i}-l_{i})+\frac{2}{\alpha}(l_{i}-y_{i})\boldsymbol{1}(y_{i}<l_{i})+\frac{2}{\alpha}(y_{i}-u_{i})\boldsymbol{1}(y_{i}>u_{i}).
\]
We choose this metric to jointly evaluate a family of intervals in
terms of precision (i.e. the width of the intervals) and accuracy
(i.e., the coverage of the true value). Notably, low values of the score indicate good performance.

\subsection{Results}

Figure \ref{fig:resultSim} summarizes the resulting RMSEs, PI coverages
and averages of the interval scores across the 1000 generated
datasets for the two scenarios.

\begin{figure}[h!]
\centering \includegraphics[width=1\textwidth]{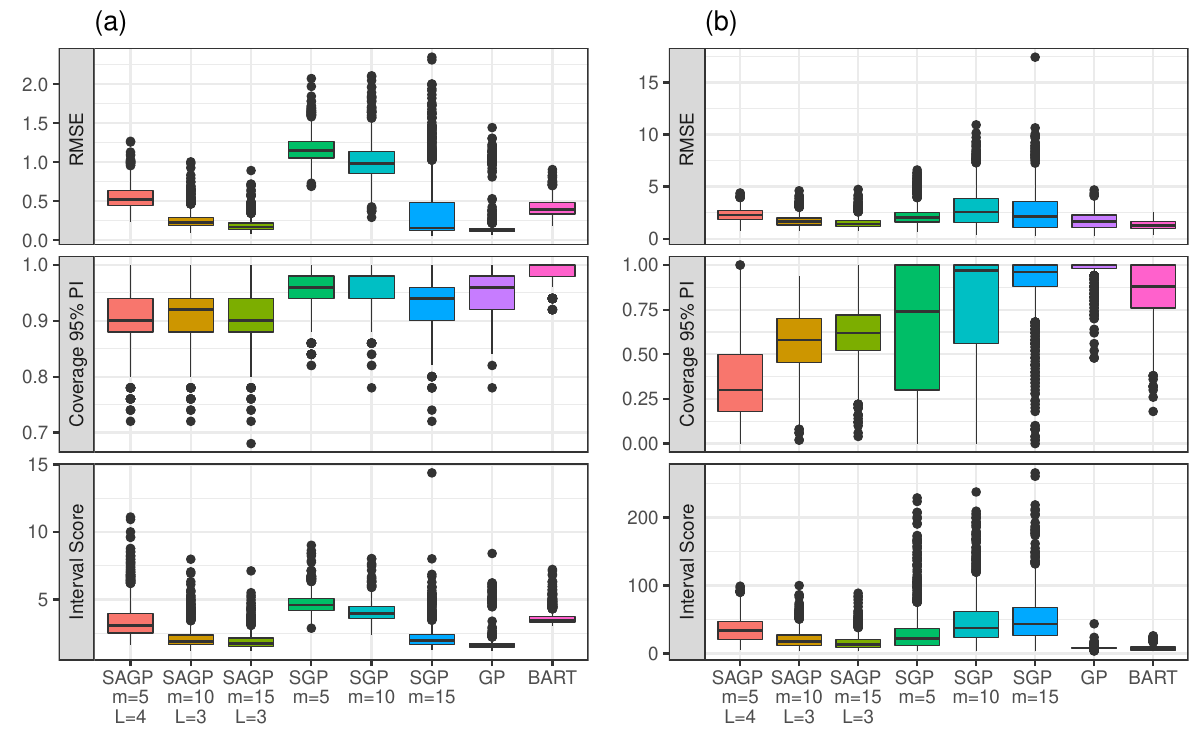}
\caption{RMSE (top panels), coverage (central panels)
and interval score (bottom panels) over the 1000 simulated datasets.
Panel (a) shows the results in the case where the testing
set is chosen at random over $[0,1]$.
Panel (b) shows the results in the case where the testing
set is chosen in a random interval with center uniformly selected from $[0.25,0.75]$. \label{fig:resultSim}}
\end{figure}

Panel (a) provides the results
in the scenario where the testing set is selected at random. In terms of RMSE,
both the SAGP and SGP models perform better with larger values of
$m$. As expected, the full GP model attains the smallest RMSEs. The
SAGP models with $m=5$ and 10 perform better than the SGP models
with the same number of pseudo-inputs. For $m=15$, the median RMSEs
in the SAGP and SGP models are similar, but the performance of the
SAGP model is more consistent across simulations (the upper quartile
of SAGP with $m=15$ is considerably smaller than the one of SGP with
$m=15$). With the considered configuration of the parameters, the
BART model performs slightly better than the SAGP model with $m=5$,
but worse than the SAGP model with $m=10$ and $m=15$. The coverage of the 95$\%$ PIs is close to the nominal level for
all the methods except for BART. The PIs of the SAGP model appear
to be slightly too narrow, as most of the coverages are a little lower
than .95. SGP and GP models show coverages perfectly matching the
nominal value. However, the BART model produces overly wide PIs, as
the median coverage is 100$\%$. The ranking of the methods in terms of interval score is similar to
the one based on the RMSE. Again, better performances are attained
by SAGP and SGP models with larger values of $m$. The SAGP models
with $m=10$ and $m=15$ perform better than all the other methods,
except for the full GP model.

Panel (b) provides the results
of the simulations in the scenario where the testing set is an interval
with random mid-point. Notably, this prediction problem is much harder
than the one evaluated by the previous scenario, as the models are
forced to a certain degree of extrapolation due to the lack of data.
BART, GP and SAGP with $m=15$ attain the best performance in terms of RMSE. Overall, the SAGP model seem to perform better than SGP. The coverage is suboptimal for all the methods, being much lower (undercoverage) than expected for SAGP and SGP with $m=5$ and higher (overcoverage) for GP. A wide range of coverages is observed for all the other methods. With respect
to the interval score, GP and BART are the methods that appear to perform best. Among the SAGP and SGP models, SAGP with $m=15$ is the best performing and competitive with GP and BART. 

\subsection{Computational Details}

As for any Bayesian model that is fit using MCMC algorithms, the convergence to the stationary distribution must be investigated also for the SAGP model. In our specific implementation, we discard the first 10,000 samples as burn-in and keep the following 1,000 samples to compute posterior estimates. We monitor the convergence of $\sigma^2_\epsilon$, sampled with Gibbs steps, and of the parameters $\eta^{(j)}$, which are sampled with Metropolis-Hastings steps with an adaptive choice of the bandwidth of the proposal distribution to control the acceptance rate. The considered SAGP models turned out to mix well and reasonably fast on the basis of trace-plots of the parameters (not shown) and the diagnostics suggested by \citet{Gelman2003}, which are provided in Appendix \ref{sec:Sim-1D diagnostic}. Notably, in our experience, similar satisfying
mixing diagnostic for the SAGP model may be achieved with much fewer steps than 10,000. 

\begin{table}[ht]
\centering
\caption{Computation time needed to fit the SAGP model on 1,000 simulated datasets on a 40-core cluster.}
\label{tab:timingSimulations}
\begin{tabular}{cccc}
\toprule 
$m$ & $L$ & Testing set & CPU time (hh:mm:ss) \tabularnewline
\midrule
\midrule 
10 & 3 & Random & 106:53:21 \tabularnewline
10 & 3 & Interval & 107:56:21\tabularnewline
\midrule 
5 & 4 & Random & 241:17:31 \tabularnewline
5 & 4 & Interval & 175:45:35\tabularnewline
\midrule 
15 & 3 & Random & 477:22:52 \tabularnewline
15 & 3 & Interval & 479:06:39\tabularnewline
\bottomrule
\end{tabular}
\end{table}

With respect to the computation time, setting the burn-in size to 10,000 and the size of posterior samples to 1,000, an SAGP model can be fit on one dataset of size $n=200$ in 3 to 5 minutes, depending
on the configuration of $m$ and $L$, using a laptop
with an Intel Core-i5 2.30GHz processor. The time that was needed to fit the model on one batch of 1,000 simulated datasets 
are  summarized in  Table \ref{tab:timingSimulations}.

\section{\label{sec:Real-world-Data-Applications}Real Data Applications}
In this section, the proposed model is applied to real data. We considered four datasets that differ in terms of sample size and number of predictors:
\begin{itemize}
    \item Heart rate data: $n=1,664$, $d=1$;
    \item Temperature data: $n=247$, $d=2$;
    \item Ice Sheet data: $n=2,226$, $d=2$;
    \item UK Housing data: $n=1,519$ and $d=8$.
\end{itemize}
The performance is evaluated quantitatively with the out-of-sample RMSE, coverage of 95$\%$ PIs and average interval score on a 25\% test set. Our model is compared to other popular methods. We considered two Bayesian models: BART and Bayesian CART (BCART) \citep{Chipman1998}, implemented in the $\mathtt{BayesTree}$ package on CRAN (version 0.3-1.4). We also considered two frequentist models: full GP and Local Approximate GP (laGP) \citep{gramacy2016lagp}, implemented in the $\mathtt{laGP}$ package on CRAN (version 1.5-5). 
For the $d=1$ and $d=2$ datasets, we also provide a qualitative assessment of the fits via graphical plots. 

\begin{table}[ht]
\label{tab:bigtable}
\vspace{-1cm}
\centering
\caption{Performance of SAGP model and of other competing methods on four datasets.}
\begin{tabular}{cccccc}
	\hline
Dataset & Model & Details & RMSE & Coverage (\%) & \makecell{Avg. Int. Score \\($\log_{10}$ scale)} \\
	\hline
\multirow{10}{*}{\makecell{Heart Rate \\ ($n$=1,664, $d$=1)}}
& SAGP & $L$=3, $m$=10 & 1.340$\times10^2$ & 88.8 & 2.800 \\
& SAGP & $L$=4, $m$=5 & 1.342$\times10^2$ & 89.5 & 2.796 \\
& GP & - & 2.727$\times10^2$ & 12.0 & 3.855 \\
& SGP & $m$=5 & 1.366$\times10^2$ & 100.0 & 4.605 \\
& SGP & $m$=15 & 1.347$\times10^2$ & 100.0 & 4.633 \\
& SGP & $m$=150 & 1.325$\times10^2$ & 100.0 & 4.720 \\
& laGP & ALC & 1.325$\times10^2$ & 91.7 & 2.777 \\
& laGP & MSPE & 1.879$\times10^2$ & 91.1 & 2.821 \\
& BART & - & 1.331$\times10^2$ & 18.1 & 3.494 \\
& BCART & - & 1.355$\times10^2$ & 94.7 & 2.751 \\
	\hline
\multirow{10}{*}{\makecell{Temperature \\ ($n$=247, $d$=2)}}
& SAGP & $L$=2, $m$=5 & 3.412$\times10^0$ & 79.7 & 1.345 \\
& SAGP & $L$=4, $m$=10 & 2.910$\times10^0$ & 77.4 & 1.356 \\
& GP & - & 3.041$\times10^0$ & 92.5 & 1.228 \\
& SGP & $m$=5 & 3.401$\times10^0$ & 100.0 & 2.010 \\
& SGP & $m$=15 & 3.345$\times10^0$ & 100.0 & 2.033 \\
& SGP & $m$=150 & 3.043$\times10^0$ & 100.0 & 1.709 \\
& laGP & ALC & 3.206$\times10^0$ & 86.6 & 1.247 \\
& laGP & MSPE & 3.431$\times10^0$ & 85.3 & 1.273 \\
& BART & - & 3.123$\times10^0$ & 52.3 & 1.658 \\
& BCART & - & 3.432$\times10^0$ & 90.8 & 1.195 \\
	\hline
\multirow{11}{*}{\makecell{Ice Sheet \\ ($n$=2,226, $d$=2)}}
& SAGP & $L$=3, $m$=10 & 1.944$\times10^2$ & 89.6 & 3.048 \\
& SAGP & $L$=3, $m$=15 & 1.858$\times10^2$ & 89.1 & 3.038 \\
& SAGP & $L$=4, $m$=5 & 2.126$\times10^2$ & 89.7 & 3.073 \\
& GP & - & 0.766$\times10^2$ & 93.8 & 2.570 \\
& SGP & $m$=5 & 2.575$\times10^2$ & 100.0 & 5.638 \\
& SGP & $m$=15 & 2.278$\times10^2$ & 100.0 & 5.841 \\
& SGP & $m$=150 & 1.637$\times10^2$ & 100.0 & 6.365 \\
& laGP & ALC & 1.672$\times10^2$ & 88.7 & 2.892 \\
& laGP & MSPE & 1.715$\times10^2$ & 88.7 & 2.894 \\
& BART & - & 1.532$\times10^2$ & 49.9 & 3.322 \\
& BCART & - & 2.231$\times10^2$ & 91.3 & 3.026 \\
	\hline
\multirow{14}{*}{\makecell{UK Budget \\ ($n$=1,519, $d$=8)}}
& SAGP & $L$=2, $m$=10 & 3.486$\times10^1$ & 92.3 & 2.327 \\
& SAGP & $L$=2, $m$=15 & 3.370$\times10^1$ & 92.8 & 2.312 \\
& SAGP & $L$=3, $m$=10 & 3.478$\times10^1$ & 92.2 & 2.327 \\
& SAGP & $L$=3, $m$=15 & 3.366$\times10^1$ & 92.7 & 2.313 \\
& GP & - & 3.105$\times10^1$ & 94.2 & 2.245 \\
& SGP & $m$=5 & 3.087$\times10^1$ & 5.0 & 2.904 \\
& SGP & $m$=15 & 3.053$\times10^1$ & 13.9 & 2.855 \\
& SGP & $m$=150 & 3.112$\times10^1$ & 49.0 & 2.647 \\
& laGP & ALC & 4.459$\times10^1$ & 55.3 & 2.679 \\
& laGP & MSPE & 4.529$\times10^1$ & 54.7 & 2.685 \\
& BART & - & 3.065$\times10^1$ & 48.1 & 2.605 \\
& BCART & - & 3.613$\times10^1$ & 92.4 & 2.286 \\
\hline
\end{tabular}
\end{table}
\subsection{Heart Rate Data}

The heart rate (HR) dataset we study here can be used to evaluate
the level of physical preparation and design training/rehabilitation
activities \citep{Zakynthinaki2015}. 
%
In this study, a single runner was asked to run on a treadmill
at constant speed. The HR (in beats/minute) was recorded for about
7 minutes from the beginning of the exercise. After the exercise,
the HR of the subject was measured for about 10 minutes during the
recovery. The experiment was repeated four times, varying the speed
of the exercise ($v=13.4,14.4,15.7$ and 17 km/h). For our illustrative
purposes, we use the data of the exercise performed at speed $v=13.4$
km/h, which are graphically represented in Figure \ref{fig:fit_HR}.

We consider SAGP models with $m=5$ and $m=10$ pseudo-inputs, and use 10-fold
cross-validation to select the number of layers $L$ as shown in Figure \ref{fig:CV_HR}(a).
This plot demonstrates the 
trade-off between the values of pseudo imputs $m$ and the number
of layers $L$, with $L=3, m=10$ being a good choice.
The resulting fitted SAGP model, which consists of 7 additive components, is
summarized in Figure \ref{fig:fit_HR}, both in terms of how the fit is decomposed by layer in panel (a) and the overall fit shown in panel (b).  An alternative fit with
$L=4, m=5$ is provided in Appendix \ref{sec:Heart-Rate-Dataset}.

\begin{figure}[t]
\centering

\includegraphics[width=0.8\textwidth]{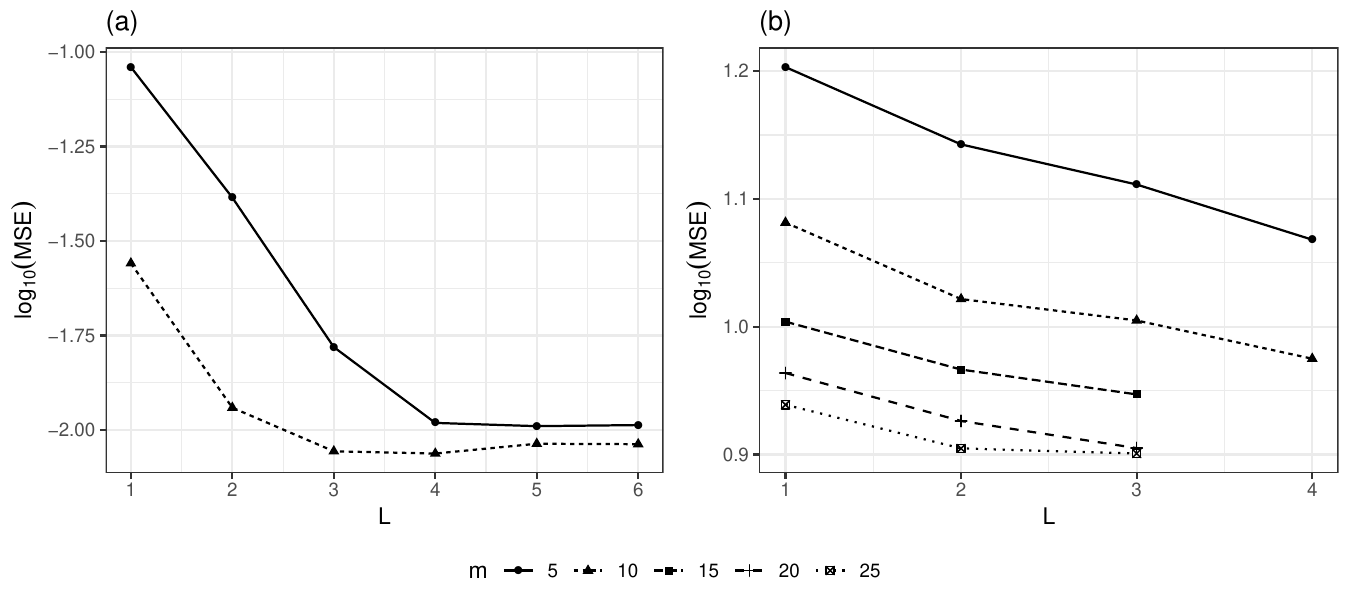} 

\caption{(a) Out-of-sample MSE (on $\log_{10}$ scale) attained by different
models fitted on 1 dimensional heart-rate dataset with $m=5,10$ and
$L=1,\ldots,6$. \newline
 (b) Out-of-sample MSE (on $\log_{10}$ scale) attained by different
models fitted on 2 dimensional temperature dataset with $m=5,10,15,20,25$
and $L=1,\ldots,4$. For any $L>4$, our pruning algorithm \ref{alg:pruning-algorithm-1}
will reduce it to $L=4$; for $m=25$ our pruning algorithm will reduce SAGP model to $L=3$
 \label{fig:CV_HR}
 }
\end{figure}

\begin{figure}[t]
\centering

\includegraphics[width=1\textwidth]{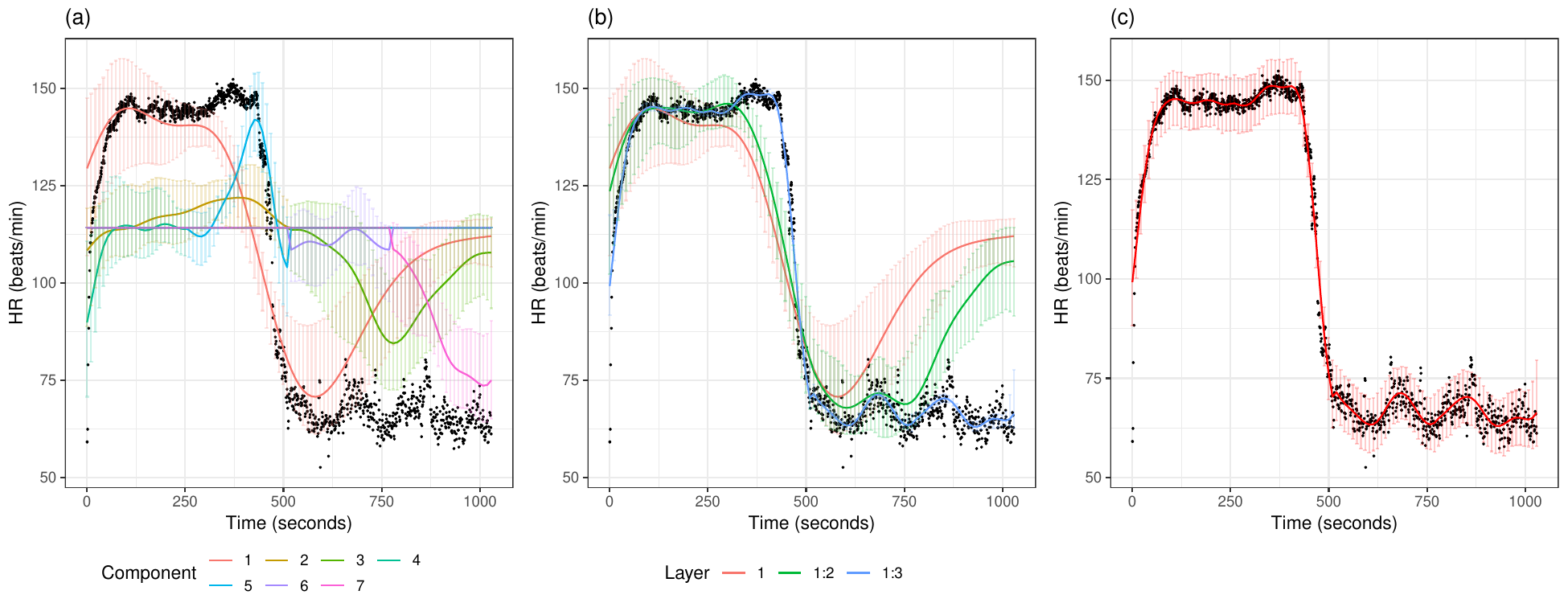} 

\caption{The panels show the observed HR values over time as black dots and
results about the fit of the SAGP model with $m=10$ and $L=3$. Panel
(a) shows the posterior means and the 95\% CIs of the 7 additive components
of the SAGP model on 100 equispaced locations on the support of the
data. Panel (b) shows the posterior means and the 95\% CIs of the
sole component in layer 1 (red), of the components belonging to layer
1 and 2 (green) and of the complete model, including components from
layer 1, 2 and 3. Panel (c) provides the predictive mean and the corresponding
95\% prediction intervals. \label{fig:fit_HR}}
\end{figure}

\subsection{Temperature Data}

In this section we study a moderate sized 2-dimensional dataset 
of average daily maximum temperature in degrees centigrade at 247 locations in Colorado
during 1997 \href{https://www.image.ucar.edu/Data/US.monthly.met/USmonthlyMet.shtml}{US precipitation and temperature (1895-1997) dataset}.
%

Qualitative comparisons of GP, BART and SAGP are shown in Figure \ref{fig:The comparison result}.
For GP regression we used MLE estimates with the
$\mathsf{Matern}(5/2)$ kernel. For BART we
use the default settings \citep{Chipman_etal2010}.
For SAGP, we choose $L=3,m=25$
and calibrate the $\alpha,\beta$ of the noise
prior in SAGP and the noise estimate in BART according to MLE of noise
estimate from GP.
The GP model shows reasonable predictions, however, the prediction comes with high
predictive variance in locations away from the observations and especially near the boundary (not shown).
The predictive mean of BART shows
it has a slight grid-like artifact due to its decision tree construction. In addition, the shape of the response around the mode is noticeably more rectangular than suggested by the other models.

This dataset provides us a 2-dimensional example where the data is
limited, which is actually a disadvantage for SAGP since the sparsification
does not cut down the computational cost significantly yet some information
is lost in the procedure. Nonetheless, the SAGP method  captures
the major trends and even some of the extremal temperatures close
to 40 degrees centigrade. Compared to BART and GP, 
the SAGP model behaves 
``in-between'' these two
methods and provides us with very competitive performance. 

\begin{figure}[t]
\centering
\includegraphics[width=5cm]{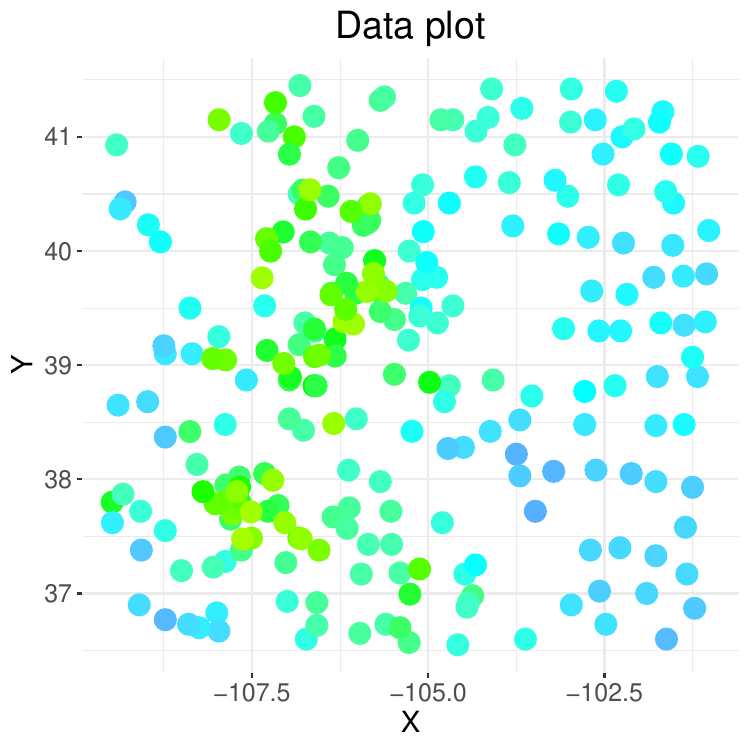}
\includegraphics[width=5cm]{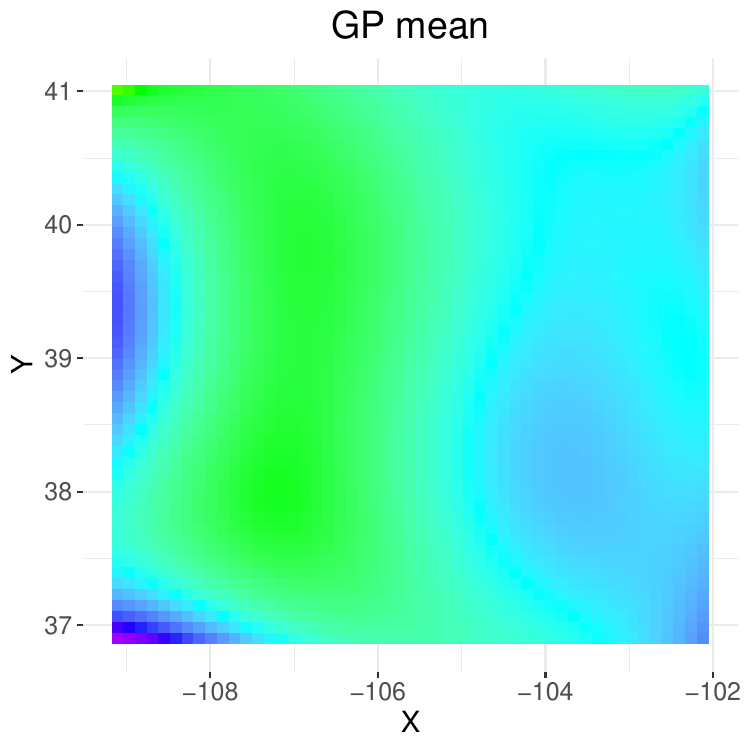}\par
\includegraphics[width=5cm]{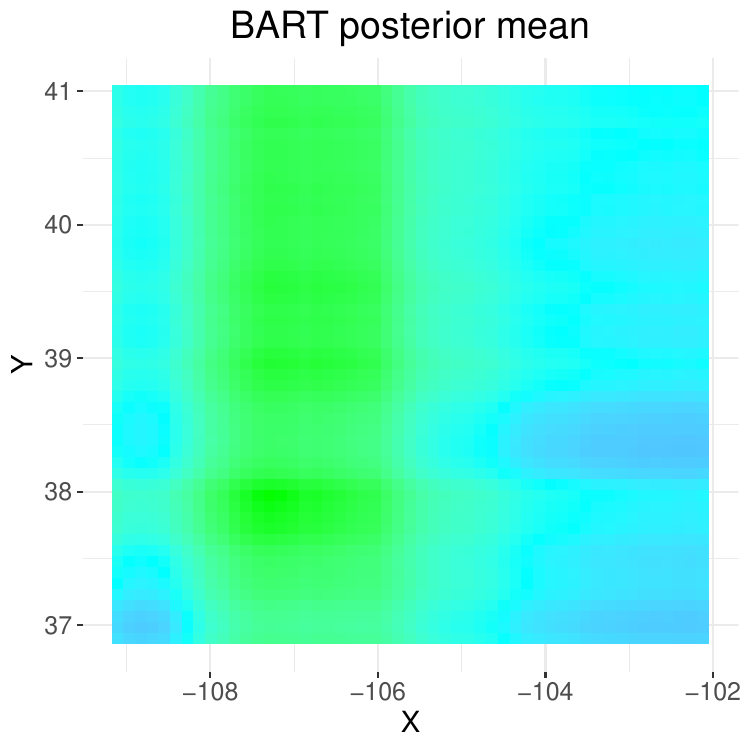} 
\includegraphics[width=5cm]{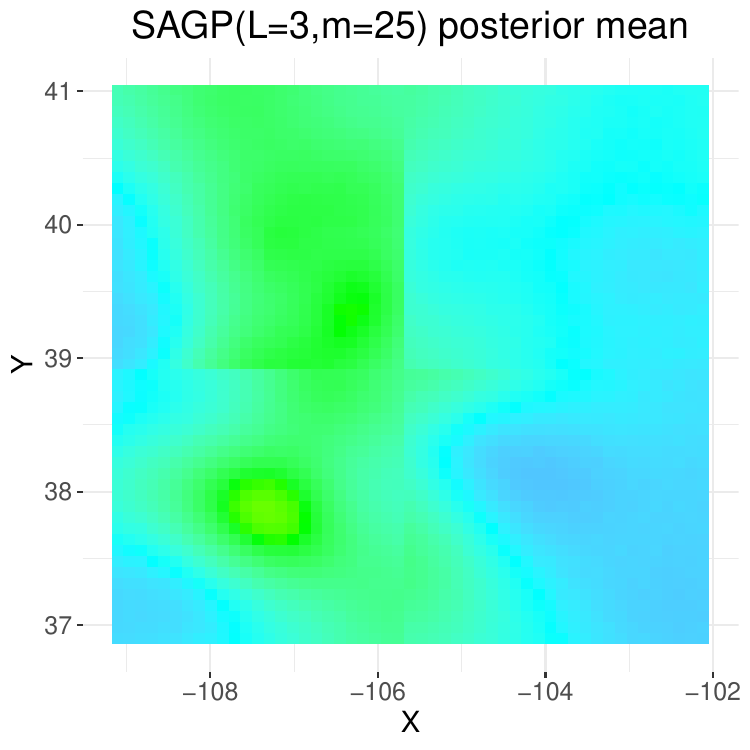}\par
\includegraphics[scale=0.6]{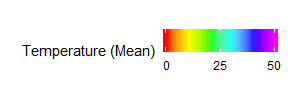}

\caption{ 
\label{fig:max_temp Original data}
\label{fig:The comparison result} 
The original max temperature dataset
for Colorado in 1997 summer. The horizontal axis is longitude;
the vertical axis is latitude; the response is the observed
values of maximal temperature in degrees Celsius.
The typical raster plot for predictive means of Gaussian Process regression (Universal kriging with Gaussian kernel and MLE nuggets)/BART(number of trees $m=200$)/SAGP($m=25,L=3$) evaluated on a fine meshed grid (generated by steplengths of 0.1) on the original input domain.
} 
\end{figure}


\subsection{Ice Sheet Data}

The Ice Sheet data is a larger 2-dimensional dataset but this time with noticeably 
uneven sampling as discussed in \citet{Park&Apley2018}.
The response is ice sheet thickness in meters collected 
over a region of west Antarctica \citep{Blankenship2004Ice}.
We used the data from 1991, first converting the longitude and latitude into
2-dimensional Euclidean coordinates and standardizing the dataset
to $[0,1]^{2}$. 

A plot of the data and predictive fits for the GP (exponential correlation), laGP, BART, treed GP (TGP; \citep{gramacy2007tgp}) and 
SAGP models are shown in Figure \ref{fig:The comparison result-1}.
We included TGP in this plot as we thought it may be helpful with the unevenly sampled data
but did not end up including it in our overall quantitative results below.
For the SAGP model, we show the fit obtained with $L=3, m=10.$


The fits obtained among these models show quite different
behaviors. The full GP fit 
possess extreme boundary behavior due to the lack of data near the boundary. The BART
model shows more noticeable grid-like artifacts in this dataset, but does not suffer from the boundary effects seen with the GP. 
The TGP regression also does not exhibit boundary effects but has much higher
variability of the mean response in the data-rich region which does not agree with the other models.
The dynamic partitioning of TGP also introduces
considerable computational cost. 
The laGP model with its default settings and MSPE criterion exhibits some degree of variability in the fitted mean response, particularly
near the boundaries,
however, it is the most computationally efficient
method. 

\begin{figure}[ht!]
\begin{centering}
\centering 
\par\end{centering}
\begin{centering}
\includegraphics[width=5cm]{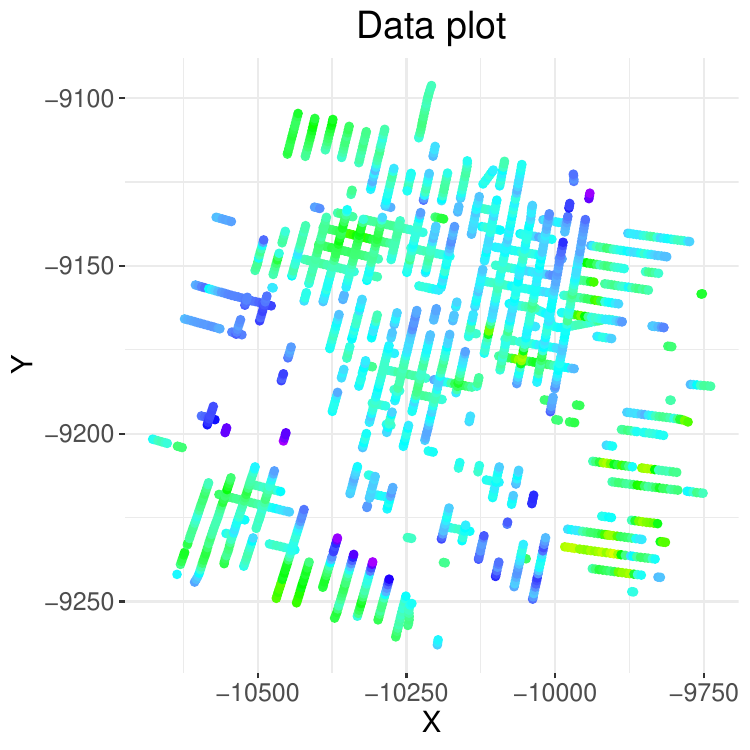}
\includegraphics[width=5cm]{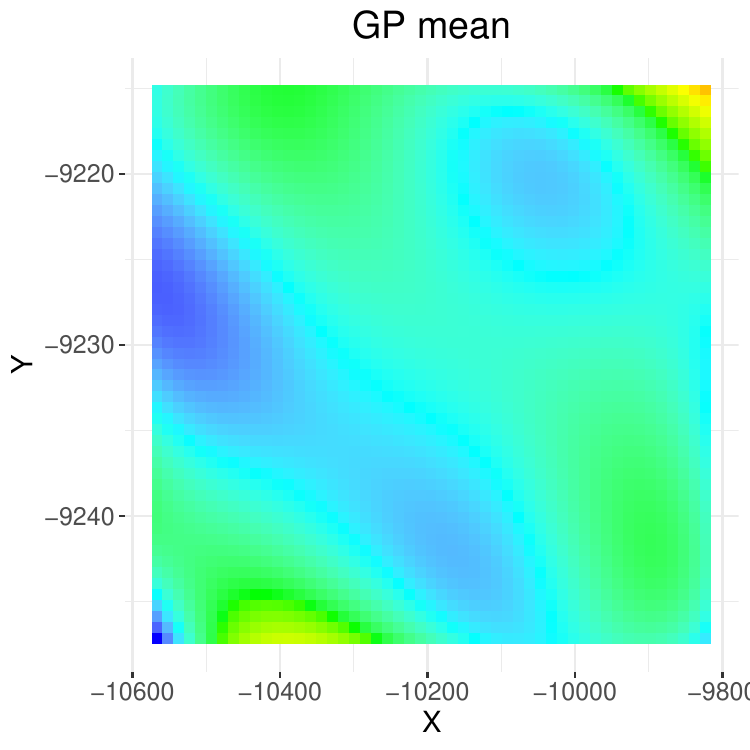}
\includegraphics[width=5cm]{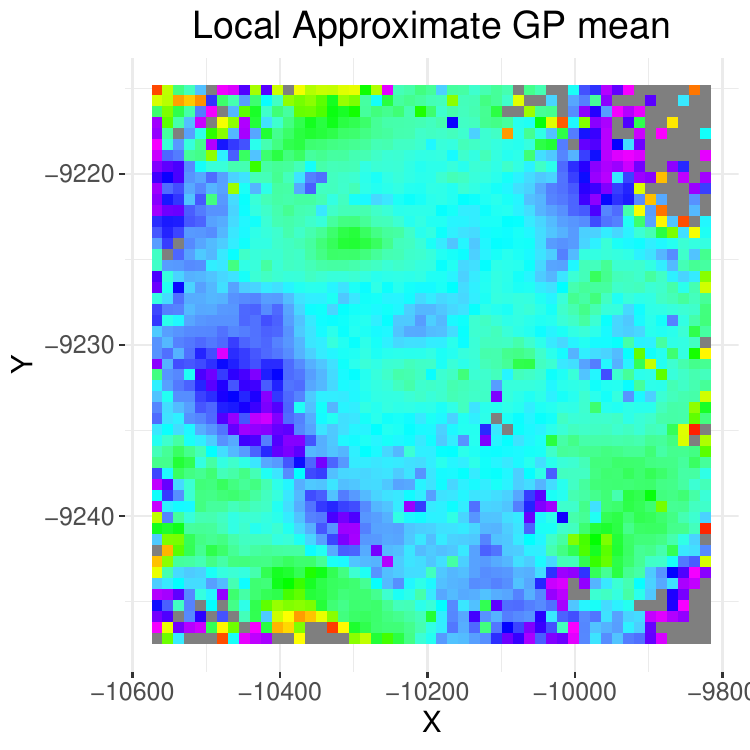}\\
\includegraphics[width=5cm]{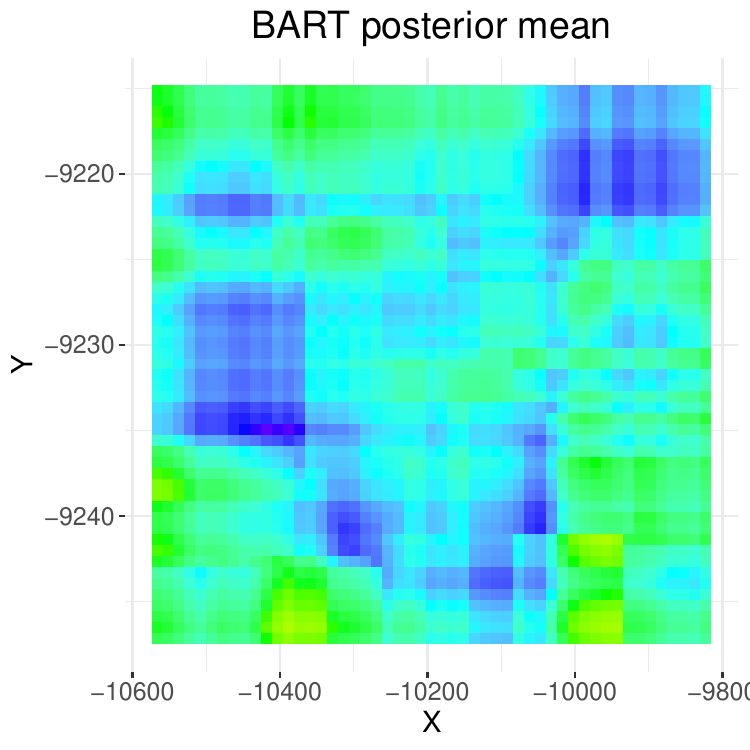}
\includegraphics[width=5cm]{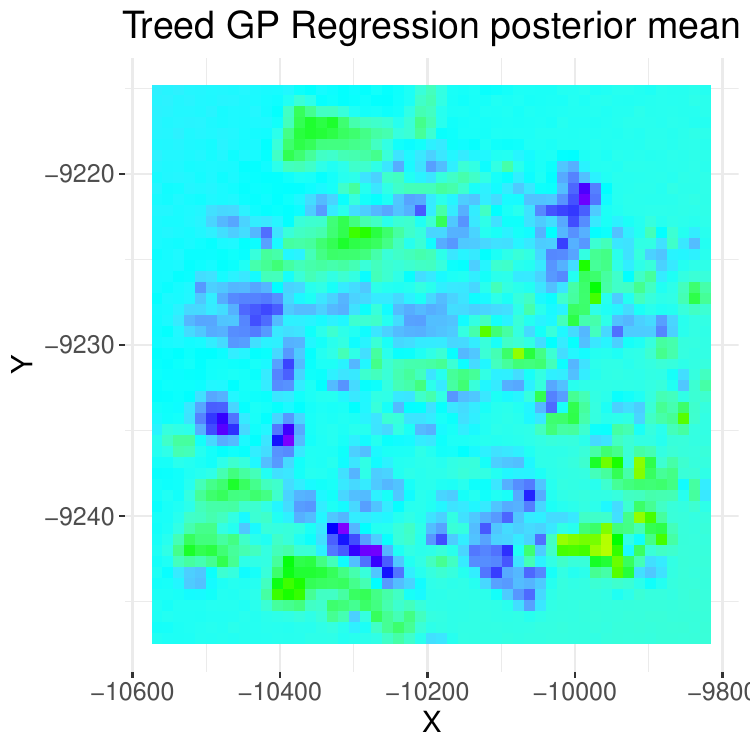}
\includegraphics[width=5cm]{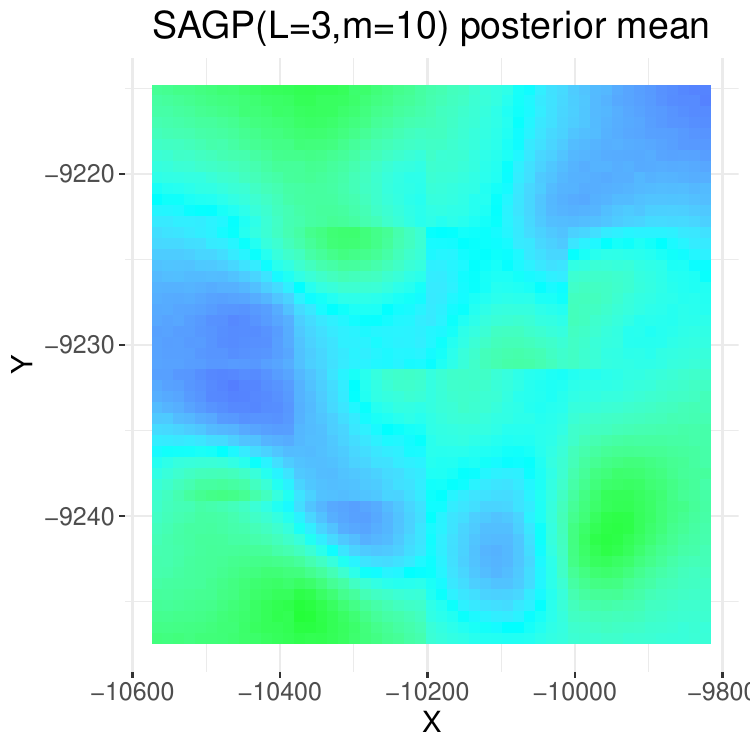}
\par
\includegraphics[scale=0.6]{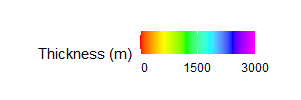}
\end{centering}
\centering{}\caption{\label{fig:The comparison result-1}\label{fig:The comparison result-1-1}The scatter point plot shows the ice thickness in a region of Antarctica. The horizontal and the vertical axis are geographical coordinates in kilometers (km).
The raster plot for predictive
means of GP/laGP/BART(number of trees $m=200$)/TGP/SAGP($m=10, L=3$) for ice sheet dataset and evaluation on a fine meshed
grid (generated by steplengths of 0.02) on the original input domain.
The color scales and the axis are the same in these plots.}
\end{figure}


\FloatBarrier

The SAGP model fit is reasonable. It does not have extreme
values, high variability in the mean response or boundary effects  like some of the other models,
yet retains most of the smoothness suggested by the full GP fit. 

\subsection{\label{subsec:Electric-Motor-Data}United Kingdom Budget Data}

This dataset is a well-known econometric dataset first studied in \citet{blundell1998semiparametric}. The dataset consists of a cross-section of 1,519 households drawn from the 1980-1982 British Family Expenditure Surveys. We attempt to predict the total household expenditure (rounded to the nearest 10 UK pounds sterling) with 8 variables as inputs. We do not use the variable of the number of children per household in the regression, since the dataset is cleaned in such a way that it contains only households with one or two children, as presented in \citet{blundell1998semiparametric}. 

We choose this 
dataset to explore the performance of our SAGP model in the higher-dimensional scenario. As mentioned by \citet{gramacy2016lagp},
such a dataset of high-dimensionality ($d=8$) will usually
present computational challenges to classical GP models. Our main
goal is to show that with reasonable increase of computational time,
SAGP model has competitive performance. 
Since this dataset cannot be easily visualized, we only present quantitative results as shown in the next section.

\subsection{\label{subsec:quantitative-results}Quantitative Performance Summary}
The performance of SAGP and the alternative models considered is summarized quantitatively in Table \ref{tab:bigtable}.
As in Section \ref{sec:Simulation Study}, we summarize the quantitative performance using 25\% test set of original dataset to calculate out-of-sample RMSE, coverage of 95\% credible intervals and interval scores.  SAGP, SGP, GP, laGP, BART and BCART models were applied to all datasets.  For SAGP, we generally selected $L=2\sim 4$ and $m=5\sim 15$ while for SGP we selected $m=5, 15$ or $150.$  BART and BCART models were fit using their defaults, and laGP was fit using defaults but with both ALC and MSPE local design criteria. 

Generally, we see that models could excel in one aspect (say RMSE) typically at the expense of another aspect of model fit quality, where the quality of fit depends on the dataset and application scenario.  
We notice that BART generally had lower coverage probability for the 95\% PI
and higher interval score. BCART had better coverage probability but generally was not the best in terms of RMSE.
For the frequentist GP, two datasets exhibited good RMSE and two exhibited weaker RMSE performance.  The GP is also less informative in terms of uncertainty
quantification than the Bayesian models we considered.
The laGP models often provided good RMSE performance, particularly with the ALC criterion, however the coverage was lower on the UK dataset.

In comparison, the SAGP model generally provided RMSE performance on par or near the best model for each dataset. The coverage also shows that SAGP models were consistent performers, especially compared to BART and laGP. We also see that SGP is uniformly worse than SAGP, often having either higher RMSE or worse coverage behavior. 
Overall, it is clear that SAGP is competitive with the best models for each dataset as summarized in Table \ref{tab:bigtable}, and we often prefer the qualitative aspect of the SAGP fits compared to some of the alternative models as demonstrated earlier.

\section{\label{sec:Conclusion}Discussion}

The SAGP model effectively borrows ideas from both sparsification
and localization.  In
particular, we divide the input domain $\mathsf{X}$ in such a way
that we can choose enough pseudo-inputs fit a sparse GP regression
within the sub-region block of the partition, which also produces
a trade-off for model parameters.
We also showed that
SAGP can achieve an effective reduction in computational cost (See Proposition
\ref{prop:For-an-RP}) since all components within a layer can effectively be fit in parallel.

As a Bayesian additive model, SAGP provides uncertainty quantification
and leads to refined posterior inference. Along the model building
process, we also exhibit how the pseudo-inputs can be sampled to capture this aspect of model uncertainty, which is ignored with the fixed pseudo-inputs of SGP. 
The RP partition scheme outlined not only serves as a localization construction 
but also guarantees adequate pseudo-inputs for this resampling are available in each SAGP model component.
As shown in the data analysis examples, the SAGP model is a competitive
candidate compared to other generalization of GP regression methods.
SAGP model can easily be generalized into higher dimensions, and our RP
partition scheme is very flexible since it carries a hierarchical
structure that allows us to analyze dataset in a multiscale way. With
the homogeneous partition in one dimension, our RP scheme is similar
to \citet{Bui&Turner2014} and \citet{Lee_etal2017}; with heterogeneous
partition in higher dimensions, our RP partition scheme is more flexible.
For example, we can use binary partitioning in the first dimension
but ternary partitioning in the second dimension. This will also preserve
the hierarchical structure and allow us to decompose the high-dimensional
data through different layers.


As for future works, there are various possible extensions of the
proposed SAGP model. In terms of generalization of our current base
model, we are interested in make the SAGP model admit different covariance
kernels and different number of pseudo-inputs in each component. It is
also of interest to extend the SAGP model to binary, count and
categorical responses. To push the computational implementation of
SAGP further and since independent sparse Gaussian process (SGP) regression
model are fitted for each local component, it is readily seen that
our model is parallelizable for efficient computation. 
Theoretically,
we would also like to see a (frequentist) consistency result for SAGP
model \citep{rockova2017posterior} and a careful analysis of the
effect of the choice of priors in this model.

\acks{The authors wish to thank the helpful feedback of the editor, associate editor and two anonymous reviewers, which helped to substantially improve the paper.  The work of M.T.P.~was supported in part by the National Science Foundation under Agreement DMS-1916231 and in part by the King Abdullah University of Science and Technology (KAUST) Office of Sponsored Research (OSR) under Award No. OSR-2018-CRG7-3800.3.
}


\newpage{}

\appendix
\section{\label{sec:plot_correlation}Correlation between Targets as Function of the Distance between Inputs}
\begin{figure}[H]
\centering
\includegraphics[width=.65\textwidth]{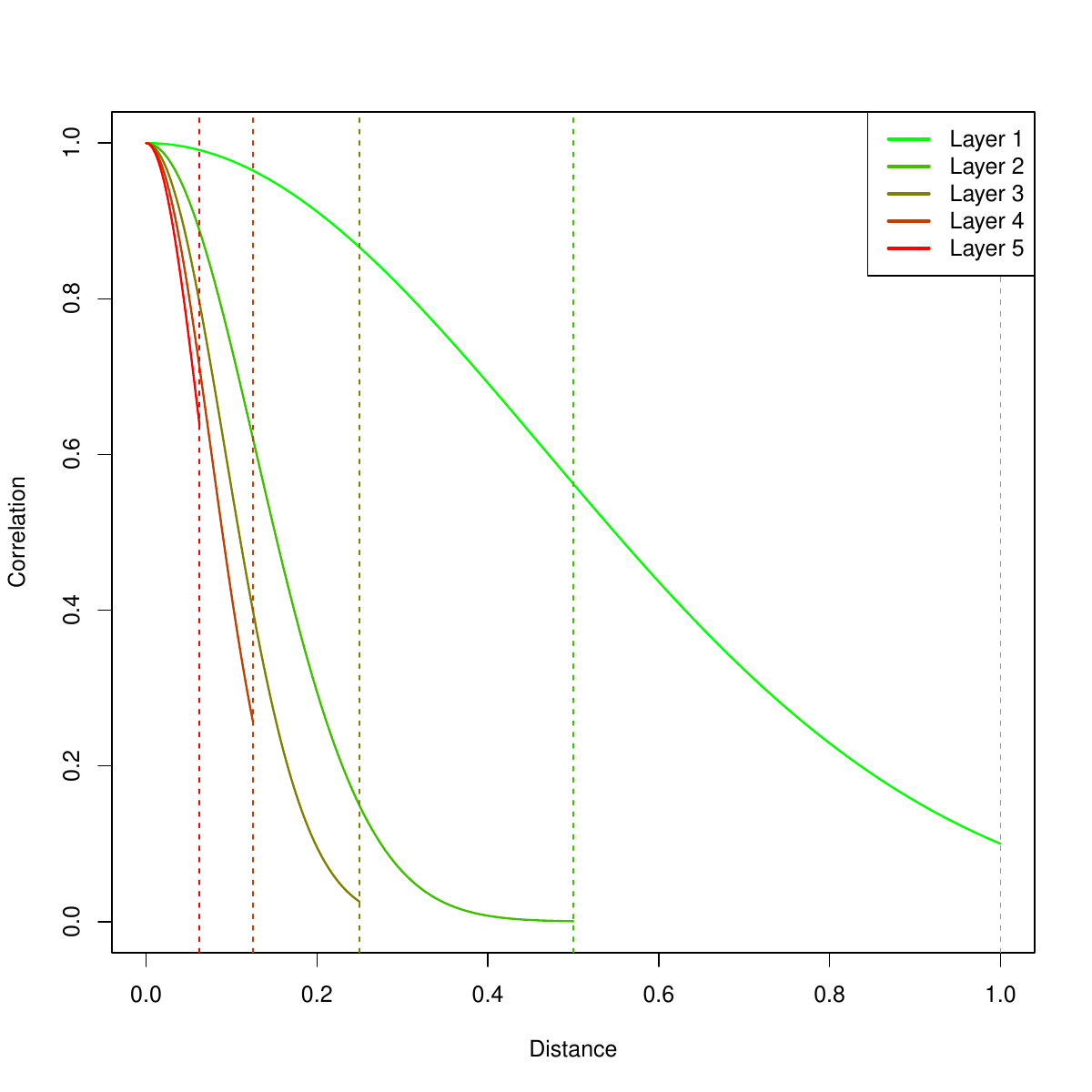}
\caption{Given two one-dimensional inputs $\boldsymbol{x}_i$ and $\boldsymbol{x}_{i'}$ ($d=1$), the figure represents the correlation between $f_j(\boldsymbol{x}_i)$ and $f_j(\boldsymbol{x}_{i'})$ (i.e., the targets of component $j$) as a function of the distance between $\boldsymbol{x}_i$ and $\boldsymbol{x}_{i'}$. We represent with different colors the correlation that is assumed for components at different layers ($L=5$ in the Figure). Notably, the size of a component's domain $B_j$ depends on the layer where the component is defined (the higher the layer index, the smaller the domain). Therefore, in our binary recursive partition scheme, the maximum possible distance  between two inputs (highlighted with a vertical dashed line in the Figure) halves at each layer. }
\end{figure}

\section{\label{sec:Proof-of-Proposition}Proof of Proposition \ref{prop:For-an-RP}}

We first calculate the total number of components in the SAGP models
when the input-domain is $[0,1]^{d}$. For each dimension $i=1,2,\ldots,d$,
if we $b_{i}$-ary subdivide the $[0,1]$ interval, then there are at most 
$|\mathcal{L}_{\ell}|=\prod_{i=1}^{d}b_{i}^{\ell-1}$ individual components
in form of $B(\boldsymbol{c}_j,\boldsymbol{w}_j)$ in the $\ell$-th layer of the
RP scheme for $\ell=1,2,\ldots,L$.

For each component in the $\ell$-th layer, the number of observations
fitted to the $j$-th component is at most $|\mathcal{X}^{(j)}|\leq n$. Then we fit a SGP model
with $m$ pseudo-inputs, whose complexity is $\mathcal{O}(|\mathcal{X}^{(j)}|\cdot m^{2})$.
Then for the $\ell$-th layer the total complexity is $\mathcal{O}(\sum_{B_j\in\mathcal{L}_{\ell}}|\mathcal{X}^{(j)}|\cdot m^{2})$. 
Therefore, for the $\ell$-th layer the complexity is at most $\mathcal{O}(|\mathcal{L}_{\ell}|\cdot n\cdot m^{2})$.
Therefore we can compute the total complexity of the model as $\mathcal{O}(\sum_{\ell=1}^{L}|\mathcal{L}_{\ell}|\cdot n\cdot m^{2})\asymp\mathcal{O}\left(\sum_{\ell=1}^{L}\prod_{i=1}^{d}b_{i}^{\ell-1}\cdot n\cdot m^{2}\right)$.


\section{\label{sec:MCMC-Algorithm-for}MCMC Algorithm for SAGP Model Fitting
(Algorithm \ref{alg:MCMC-algorithm})}

\begin{algorithm}[H]
\caption{\label{alg:MCMC-algorithm}MCMC algorithm for SAGP model.}

\setcounter{AlgoLine}{0} \SetKwData{This}{this} \SetKwFunction{AdjustBW}{Adjust
Proposal Bandwidth} \SetKwFunction{SGP}{Posterior} \SetKwFunction{Lik}{Model
Likelihood} \SetKwInOut{Input}{Input}\SetKwInOut{Output}{Output}

\Input{ RP partition scheme consisting of $N$ components, \; Number
of pseudo-inputs for each component $m_{j}$, \; Hyper-parameters
for the prior of parameters, \; Observed dataset $\{\mathcal{X},\boldsymbol{y}\}$.
\; } \Output{ Posterior samples for parameters, $\bar{\mathcal{X}}^{(j)},\bar{\boldsymbol{f}}_{j}$,
\; Predictive posterior samples for $\boldsymbol{y}_{*},\boldsymbol{f}_{j},\boldsymbol{y}$.
\; } \BlankLine 
Initialization of the parameter values\;

\While{not converged}{ Sample $\bar{\mathcal{X}}^{(j)},j=1,\ldots,N$
as in Algorithm \ref{alg:PI sampling}.


\For{$j$ in $1:N$}{ $\boldsymbol{r}_{j}\leftarrow\boldsymbol{y}-\sum_{l\neq j}\boldsymbol{K}_{nm_{l}}^{(l)}\left(\boldsymbol{K}_{m_{l}}^{(l)}\right)^{-1}\bar{\boldsymbol{f}}_{l}$\;

$\bar{\boldsymbol{f}}_{j}\leftarrow N_{m_{j}}(\boldsymbol{Mean}_{j},\boldsymbol{Var}_{j})$
as in (\ref{eq:PT posterior closed form})\;


$\eta_{j,\text{new}}\leftarrow\text{Uniform}(\eta_{j}\pm\text{bandwidth})$\;
\label{alg:ProposeEta}

$\alpha\leftarrow\min(1,{C\cdot\Lik(\eta_{j,new})}/{C\cdot\Lik(\eta_{j})})$\;


\If{Uniform(0,1)$\leq\alpha$}{ $\eta_{j}\leftarrow\eta_{j,\text{new}}$
} \tcp{For every $\text{burn-in steps}/20$ steps, we adjust bandwidth.}
\If{$\text{Acceptance rate of }\eta_{j}\notin(0.39,0.49]$}{ \label{alg:AdjBWstart}

Band width for proposing $\eta_{j}\leftarrow\text{Acceptance rate of }\eta_{j}/0.44$
\; \label{alg:AdjBW1}

}\label{alg:AdjBWend} }

$\sigma_{\epsilon}^{2}\leftarrow\text{InverseGamma}\left(\alpha_{\epsilon}+\frac{n}{2},\beta_{\epsilon}+\frac{1}{2}(\boldsymbol{y}-\hat{\boldsymbol{y}})^{T}(\boldsymbol{y}-\hat{\boldsymbol{y}})\right)$\;

} 
\end{algorithm}


\section{\label{sec:Detailed Derivation of Posterior Distribution} Detailed
Derivation of Posterior Distribution in Section \ref{sec:Posterior for PTs}}

To clarify our derivations, we first stated following simple lemma
that will be used, which can be derived from Woodbury identity \citep{Horn_Johnson1998}
or a direct verification \citep{Rasmuessen&Williams2006}. \begin{lem}
\label{lem:Conditional GP closed form Lemma}For a joint Gaussian
distribution $\boldsymbol{a}\in\mathbb{R}^{n},\boldsymbol{b}\in\mathbb{R}^{n}$
if 
\begin{align}
\left(\begin{array}{c}
\boldsymbol{a}\\
\boldsymbol{b}
\end{array}\right) & \propto N_{n+m}\left(\left(\begin{array}{c}
\boldsymbol{\mu_{a}}\\
\boldsymbol{\mu_{b}}
\end{array}\right),\left(\begin{array}{cc}
C_{\boldsymbol{aa}} & C_{\boldsymbol{ab}}\\
C_{\boldsymbol{ba}} & C_{\boldsymbol{bb}}
\end{array}\right)\right)
\end{align}
then its conditional distribution is: 
\begin{align}
\boldsymbol{a}\mid\boldsymbol{b} & \sim N_{n}\left(\boldsymbol{\mu}_{\boldsymbol{a}}+C_{\boldsymbol{ab}}\left(C_{\boldsymbol{bb}}\right)^{-1}(\boldsymbol{b-\boldsymbol{\mu}_{b}}),C_{\boldsymbol{aa}}-C_{\boldsymbol{ab}}\left(C_{\boldsymbol{bb}}\right)^{-1}C_{\boldsymbol{ba}}\right)
\end{align}

In particular, for $f_{l}=f(\boldsymbol{x}_{l}),\bar{\boldsymbol{f}}_{j}=(\bar{f}_{j}(\bar{\boldsymbol{x}}_{1}),\ldots,\bar{f}_{j}(\bar{\boldsymbol{x}}_{m_{j}}))^{T}$
and covariance kernel function $K=K^{(j)},K_{ll}^{(j)}=K{}^{(j)}(\boldsymbol{x}_{l},\boldsymbol{x}_{l})$
if 
\begin{align}
\left.\left(\begin{array}{c}
f_{l}\\
\bar{\boldsymbol{f}}_{j}
\end{array}\right)\right|\bar{\mathcal{X}}^{(j)},\boldsymbol{x}_{l} & \propto N_{1+m_{j}}\left(\left(\begin{array}{c}
0\\
\boldsymbol{0}_{m_{j}}
\end{array}\right),\left(\begin{array}{cc}
K_{ll}^{(j)} & \boldsymbol{k}_{l}^{(j)T}\\
\boldsymbol{k}_{l}^{(j)} & \boldsymbol{K}_{m_{j}}^{(j)}
\end{array}\right)\right)
\end{align}
then its conditional distribution is: 
\begin{align}
f_{l}\mid\bar{\boldsymbol{f}}_{j},\bar{\mathcal{X}}^{(j)},\boldsymbol{x}_{l} & \sim N_{1}\left(\boldsymbol{k}_{l}^{(j)T}\left(\boldsymbol{K}_{m_{j}}^{(j)}\right)^{-1}\bar{\boldsymbol{f}}_{j},K_{ll}^{(j)}-\boldsymbol{k}_{l}^{(j)T}\left(\boldsymbol{K}_{m_{j}}^{(j)}\right)^{-1}\boldsymbol{k}_{l}^{(j)}\right)
\end{align}
\end{lem}

We assume a Gaussian prior on the pseudo-targets as in (\ref{sec:Conditional Likelihood of PTs}).
\begin{align}
P(\bar{\boldsymbol{f}}_{j}\mid\bar{\mathcal{X}}^{(j)}) & \sim N_{m_{j}}\left(\bar{\boldsymbol{f}}_{j}\mid\boldsymbol{0}_{m_{j}},\boldsymbol{K}_{m_{j}}^{(j)}\right)
\end{align}
and then use Bayesian rule on the parameter $\bar{\boldsymbol{f}}_{j}$,
recalling that (\ref{eq:SAGP model}) determines the form of mean
and variance of the Gaussian distribution $P(\boldsymbol{y}\mid\bar{\boldsymbol{f}}_{1},\ldots,\bar{\boldsymbol{f}}_{N},\bar{\mathcal{X}}^{(1)},\ldots,\bar{\mathcal{X}}^{(N)},\mathcal{X},\boldsymbol{\boldsymbol{\kappa}})$.
\\
 
\begin{align}
 & P(\bar{\boldsymbol{f}}_{j}\mid\boldsymbol{y},\mathcal{X},\bar{\boldsymbol{f}}_{1},\ldots,\bar{\boldsymbol{f}}_{N},\bar{\mathcal{X}}^{(1)},\ldots,\bar{\mathcal{X}}^{(N)},\boldsymbol{\boldsymbol{\kappa}})\nonumber \\
\propto & P(\boldsymbol{y}\mid\bar{\boldsymbol{f}}_{1},\ldots\bar{\boldsymbol{f}}_{N},\bar{\mathcal{X}}^{(1)},\ldots,\bar{\mathcal{X}}^{(N)},\boldsymbol{\boldsymbol{\kappa}})\times\nonumber \\
 & P(\bar{\boldsymbol{f}}_{j}\mid\{\boldsymbol{x}\}_{n},\bar{\boldsymbol{f}}_{1},\ldots,\bar{\boldsymbol{f}}_{j-1},\bar{\boldsymbol{f}}_{j+1},\ldots\bar{\boldsymbol{f}}_{N},\bar{\mathcal{X}}^{(1)},\ldots,\bar{\mathcal{X}}^{(N)},\boldsymbol{\boldsymbol{\kappa}})\\
= & N_{n}\left(\left.\boldsymbol{y}-\sum_{l\neq j}\boldsymbol{K}_{nm_{l}}^{(l)}\left(\boldsymbol{K}_{m_{l}}^{(l)}\right)^{-1}\bar{\boldsymbol{f}}_{l}\right|\boldsymbol{K}_{nm_{j}}^{(j)}\left(\boldsymbol{K}_{m_{j}}^{(j)}\right)^{-1}\bar{\boldsymbol{f}}_{j},\boldsymbol{\Lambda}_{n}^{(j)}+\sigma_{\epsilon}^{2}\boldsymbol{I}_{n}\right)\times\nonumber \\
 & N_{m_{j}}\left(\bar{\boldsymbol{f}}_{j}\vert\boldsymbol{0}_{m_{j}},\boldsymbol{K}_{m_{j}}^{(j)}\right)
\end{align}
We can derive the posterior using the normal normal conjugacy: 
\begin{align}
 & P(\bar{\boldsymbol{f}}_{j}\mid\boldsymbol{y},\mathcal{X},\bar{\boldsymbol{f}}_{1},\ldots,\bar{\boldsymbol{f}}_{j-1},\bar{\boldsymbol{f}}_{j+1},\ldots\bar{\boldsymbol{f}}_{N},\bar{\mathcal{X}}^{(1)},\ldots,\bar{\mathcal{X}}^{(N)},\boldsymbol{\boldsymbol{\kappa}})\nonumber \\
\propto & \frac{1}{\sqrt{\left|2\pi\left(\boldsymbol{\Lambda}_{n}^{(j)}+\sigma_{\epsilon}^{2}\boldsymbol{I}_{n}\right)\right|}}\exp\left\{ -\frac{1}{2}\left[\boldsymbol{y}-\sum_{l=1}^{N}\boldsymbol{K}_{nm_{l}}^{(l)}\left(\boldsymbol{K}_{m_{l}}^{(l)}\right)^{-1}\bar{\boldsymbol{f}}_{l}\right]^{T}\left(\boldsymbol{\Lambda}_{n}^{(j)}+\sigma_{\epsilon}^{2}\boldsymbol{I}_{n}\right)^{-1}\times\right.\nonumber \\
 & \left.\left[\boldsymbol{y}-\sum_{l=1}^{N}\boldsymbol{K}_{nm_{l}}^{(l)}\left(\boldsymbol{K}_{m_{l}}^{(l)}\right)^{-1}\bar{\boldsymbol{f}}_{l}\right]\right\} \times\nonumber \\
 & \frac{1}{\sqrt{\left|2\pi\boldsymbol{K}_{m_{j}}\right|}}\exp\left\{ -\frac{1}{2}\bar{\boldsymbol{f}}_{j}^{T}\boldsymbol{K}_{m_{j}}^{-1}\bar{\boldsymbol{f}}_{j}\right\} \\
 & \text{We complete the squares inside the exponent,}\nonumber \\
\propto & \exp\left\{ -\frac{1}{2}\bar{\boldsymbol{f}}_{j}^{T}\left(\boldsymbol{K}_{m_{j}}^{-1}+\left[\left(\boldsymbol{K}_{m_{j}}^{(j)}\right)^{-1}\boldsymbol{K}_{m_{j}n}^{(j)}\left(\boldsymbol{\Lambda}_{n}^{(j)}+\sigma_{\epsilon}^{2}\boldsymbol{I}_{n}\right)^{-1}\boldsymbol{K}_{nm_{j}}^{(j)}\left(\boldsymbol{K}_{m_{j}}^{(j)}\right)^{-1}\right]\right)\bar{\boldsymbol{f}}_{j}\right.\nonumber \\
 & \left.-\boldsymbol{f}_{j}^{T}\left(\boldsymbol{\Lambda}_{n}^{(j)}+\sigma_{\epsilon}^{2}\boldsymbol{I}_{n}\right)^{-1}\boldsymbol{K}_{nm_{j}}^{(j)}\left(\boldsymbol{K}_{m_{j}}^{(j)}\right)^{-1}\bar{\boldsymbol{f}}_{j}+\text{proportionally constant terms}\right\} 
\end{align}
After completing square we can obtain the mean and variance of the
$j$-th component: 
\begin{align}
\boldsymbol{Mean}_{j}= & \left(\left(\boldsymbol{K}_{m_{j}}^{(j)}\right)^{-1}+\left[\left(\boldsymbol{K}_{m_{j}}^{(j)}\right)^{-1}\boldsymbol{K}_{m_{j}n}^{(j)}\left(\boldsymbol{\Lambda}_{n}^{(j)}+\sigma_{\epsilon}^{2}\boldsymbol{I}_{n}\right)^{-1}\boldsymbol{K}_{nm_{j}}^{(j)}\left(\boldsymbol{K}_{m_{j}}^{(j)}\right)^{-1}\right]\right)\times\nonumber \\
 & \left(\left(\boldsymbol{y}-\sum_{l\neq j}\boldsymbol{K}_{nm_{l}}^{(l)}\left(\boldsymbol{K}_{m_{l}}^{(l)}\right)^{-1}\bar{\boldsymbol{f}}_{l}\right)^{T}\left(\boldsymbol{\Lambda}_{n}^{(j)}+\sigma_{\epsilon}^{2}\boldsymbol{I}_{n}\right)^{-1}\boldsymbol{K}_{nm_{j}}^{(j)}\left(\boldsymbol{K}_{m_{j}}^{(j)}\right)^{-1}\right)^{T}\nonumber \\
= & \boldsymbol{K}_{m_{j}}^{(j)}\boldsymbol{Q}_{m_{j}}^{(j)-1}\boldsymbol{K}_{m_{j}n}^{(j)}\left(\boldsymbol{\Lambda}_{n}^{(j)}+\sigma_{\epsilon}^{2}\boldsymbol{I}_{n}\right)^{-1}\left(\boldsymbol{y}-\sum_{l\neq j}\boldsymbol{K}_{nm_{l}}^{(l)}\left(\boldsymbol{K}_{m_{l}}^{(l)}\right)^{-1}\bar{\boldsymbol{f}}_{l}\right)\label{eq:Mean_j}
\end{align}
By Woodbury identity, we know that for $\boldsymbol{Q}_{m_{j}}^{(j)}=\boldsymbol{K}_{m_{j}}^{(j)}+\boldsymbol{K}_{m_{j}n}^{(j)}\left(\boldsymbol{\Lambda}_{n}^{(j)}+\sigma_{\epsilon}^{2}\boldsymbol{I}_{n}\right)^{-1}\boldsymbol{K}_{nm_{j}}^{(j)}$
we can write its inverse as 
\begin{align*}
\boldsymbol{Q}_{m_{j}}^{(j)-1} & =\left\{ \boldsymbol{K}_{m_{j}}^{(j)-1}-\boldsymbol{K}_{m_{j}}^{(j)-1}\boldsymbol{K}_{m_{j}n}^{(j)}\left[\left(\boldsymbol{\Lambda}_{n}^{(j)}+\sigma_{\epsilon}^{2}\boldsymbol{I}_{n}\right)+\boldsymbol{K}_{nm_{j}}^{(j)}\boldsymbol{K}_{m_{j}}^{(j)-1}\boldsymbol{K}_{m_{j}n}^{(j)}\right]^{-1}\boldsymbol{K}_{nm_{j}}^{(j)}\boldsymbol{K}_{m_{j}}^{(j)-1}\right\} 
\end{align*}
Using this $m_{j}\times m_{j}$ matrix $\boldsymbol{Q}_{m_{j}}$,
we can write down the covariance matrix $\boldsymbol{Var}_{j}$: 
\begin{align}
\boldsymbol{Var}_{j}= & \left(\left(\boldsymbol{K}_{m_{j}}^{(j)}\right)^{-1}+\left[\left(\boldsymbol{K}_{m_{j}}^{(j)}\right)^{-1}\boldsymbol{K}_{m_{j}n}^{(j)}\left(\boldsymbol{\Lambda}_{n}^{(j)}+\sigma_{\epsilon}^{2}\boldsymbol{I}_{n}\right)^{-1}\boldsymbol{K}_{nm_{j}}^{(j)}\left(\boldsymbol{K}_{m_{j}}^{(j)}\right)^{-1}\right]\right)^{-1}\\
= & \boldsymbol{K}_{m_{j}}^{(j)}-\boldsymbol{K}_{m_{j}}^{(j)}\left[\left(\boldsymbol{K}_{m_{j}}^{(j)}\right)^{-1}\boldsymbol{K}_{m_{j}n}^{(j)}\right]\times\nonumber \\
 & \left[\left(\boldsymbol{\Lambda}_{n}^{(j)}+\sigma_{\epsilon}^{2}\boldsymbol{I}_{n}\right)+\boldsymbol{K}_{nm_{j}}^{(j)}\left(\boldsymbol{K}_{m_{j}}^{(j)}\right)^{-1}\boldsymbol{K}_{m_{j}}^{(j)}\left(\boldsymbol{K}_{m_{j}}^{(j)}\right)^{-1}\boldsymbol{K}_{m_{j}n}^{(j)}\right]^{-1}\times\nonumber \\
 & \boldsymbol{K}_{nm_{j}}^{(j)}\left(\boldsymbol{K}_{m_{j}}^{(j)}\right)^{-1}\boldsymbol{K}_{m_{j}}^{(j)}\\
= & \boldsymbol{K}_{m_{j}}^{(j)}-\boldsymbol{K}_{m_{j}n}^{(j)}\left[\left(\boldsymbol{\Lambda}_{n}^{(j)}+\sigma_{\epsilon}^{2}\boldsymbol{I}_{n}\right)+\boldsymbol{K}_{nm_{j}}^{(j)}\left(\boldsymbol{K}_{m_{j}}^{(j)}\right)^{-1}\boldsymbol{K}_{m_{j}n}^{(j)}\right]^{-1}\boldsymbol{K}_{nm_{j}}^{(j)}\nonumber \\
= & \boldsymbol{K}_{m_{j}}^{(j)}\left\{ \left(\boldsymbol{K}_{m_{j}}^{(j)}\right)^{-1}-\right.\nonumber \\
 & \left.\left(\boldsymbol{K}_{m_{j}}^{(j)}\right)^{-1}\boldsymbol{K}_{m_{j}n}^{(j)}\left[\left(\boldsymbol{\Lambda}_{n}^{(j)}+\sigma_{\epsilon}^{2}\boldsymbol{I}_{n}\right)+\boldsymbol{K}_{nm_{j}}^{(j)}\left(\boldsymbol{K}_{m_{j}}^{(j)}\right)^{-1}\boldsymbol{K}_{m_{j}n}^{(j)}\right]^{-1}\left(\boldsymbol{K}_{m_{j}}^{(j)}\right)^{-1}\right\} \boldsymbol{K}_{nm_{j}}^{(j)}\\
= & \boldsymbol{K}_{m_{j}}^{(j)}\boldsymbol{Q}_{m_{j}}^{(j)-1}\boldsymbol{K}_{m_{j}}^{(j)}\label{eq:Var_j}
\end{align}
Note that although we do need to invert an $n\times n$ matrix $\boldsymbol{\Lambda}_{n}^{(j)}+\sigma_{\epsilon}^{2}\boldsymbol{I}_{n}$,
it is a diagonal matrix and hence easy to invert as claimed before.


\section{\label{sec:Sim-1D diagnostic}Diagnostic Statistics for the SAGP
Model on 1000 Batches of Simulated Dataset}
\begin{figure}[H]
\centering
\includegraphics[width=7cm,height=14cm]{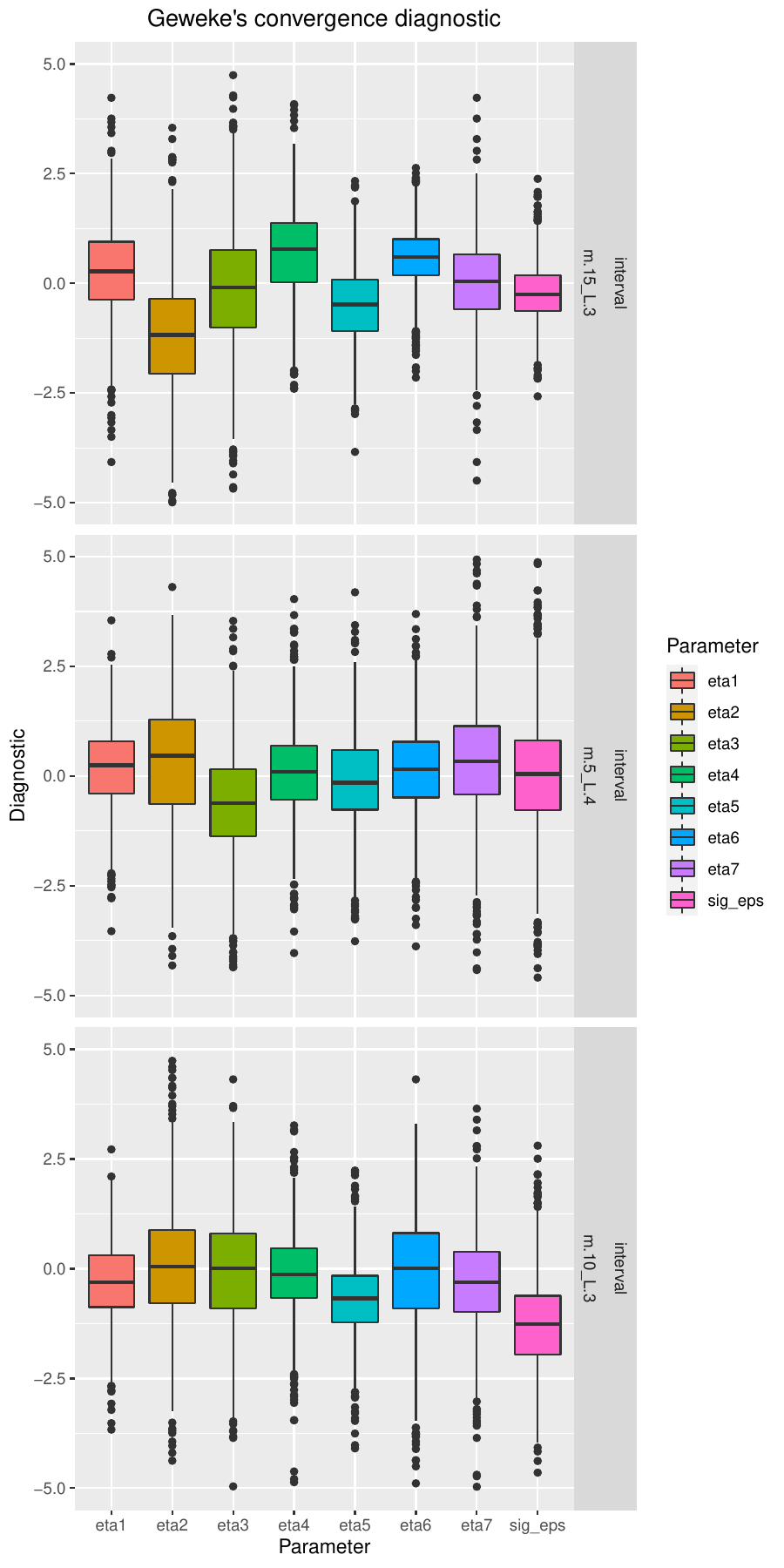}\includegraphics[width=7cm,height=14cm]{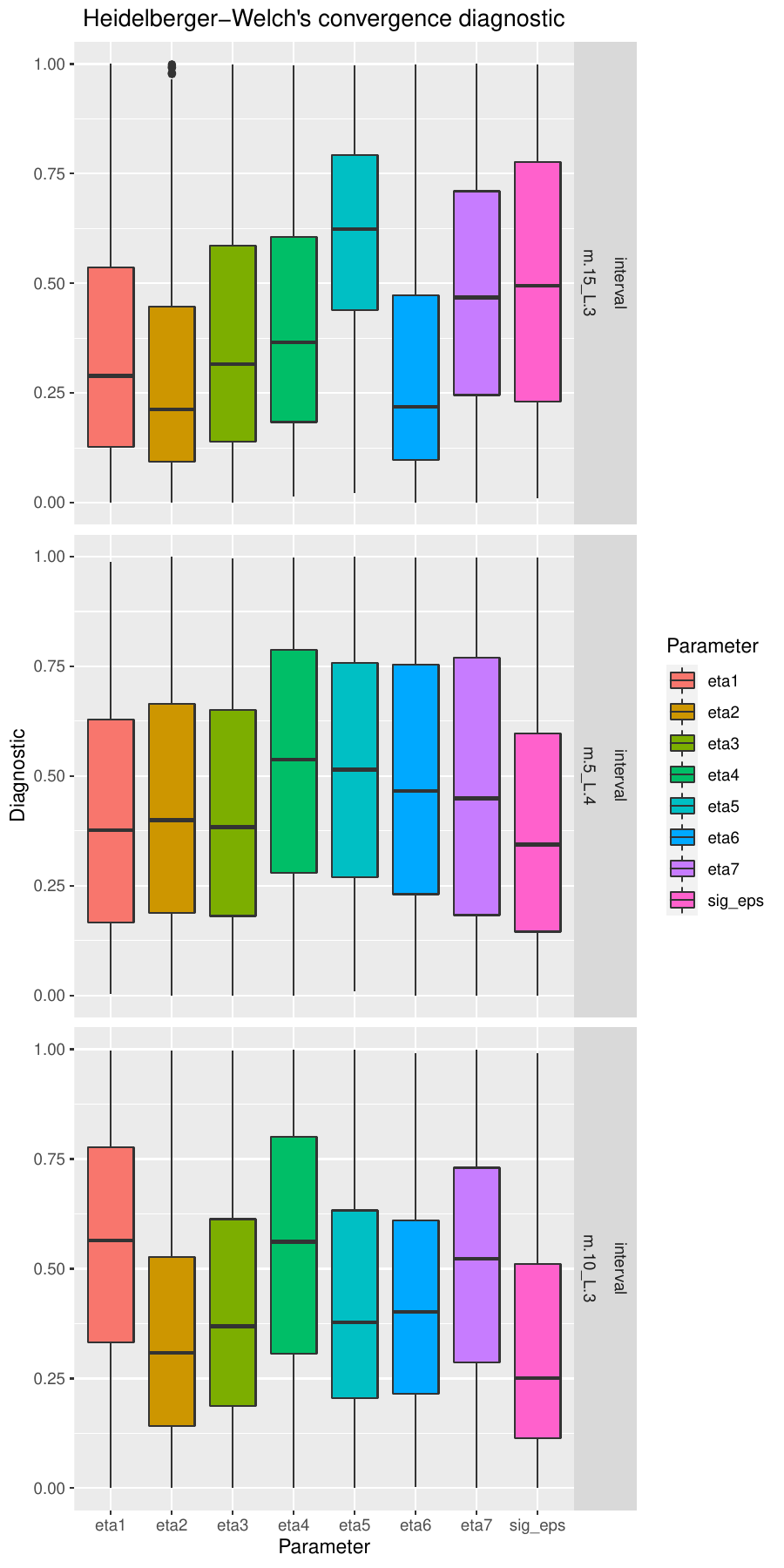}

\caption{The panels of box plots show the Geweke's convergence diagnostic \citep{geweke1991evaluating}
and Heidelberger-Welch's convergence diagnostic \citep{heidelberger1983simulation}
based on the MCMC sample of SAGP model, for parameter $\eta^{(j)}$
and $\sigma_{\epsilon}^{2}$, calculated from the 1000 batches of
simulated dataset from formula (\ref{eq:simulated dta formula}) with
the testing set is random or interval. \label{fig:diagnostic_random}}
\end{figure}
\FloatBarrier

\section{\label{sec:Heart-Rate-Dataset}Heart Rate Dataset Analyzed by SAGP
Model Fitted with \texorpdfstring{$m=5$}{m=5} and \texorpdfstring{$L=4$}{L=4} (Figure \ref{fig:fit_HR-1})}

\FloatBarrier

\begin{figure}[h!]
\includegraphics[width=1\textwidth]{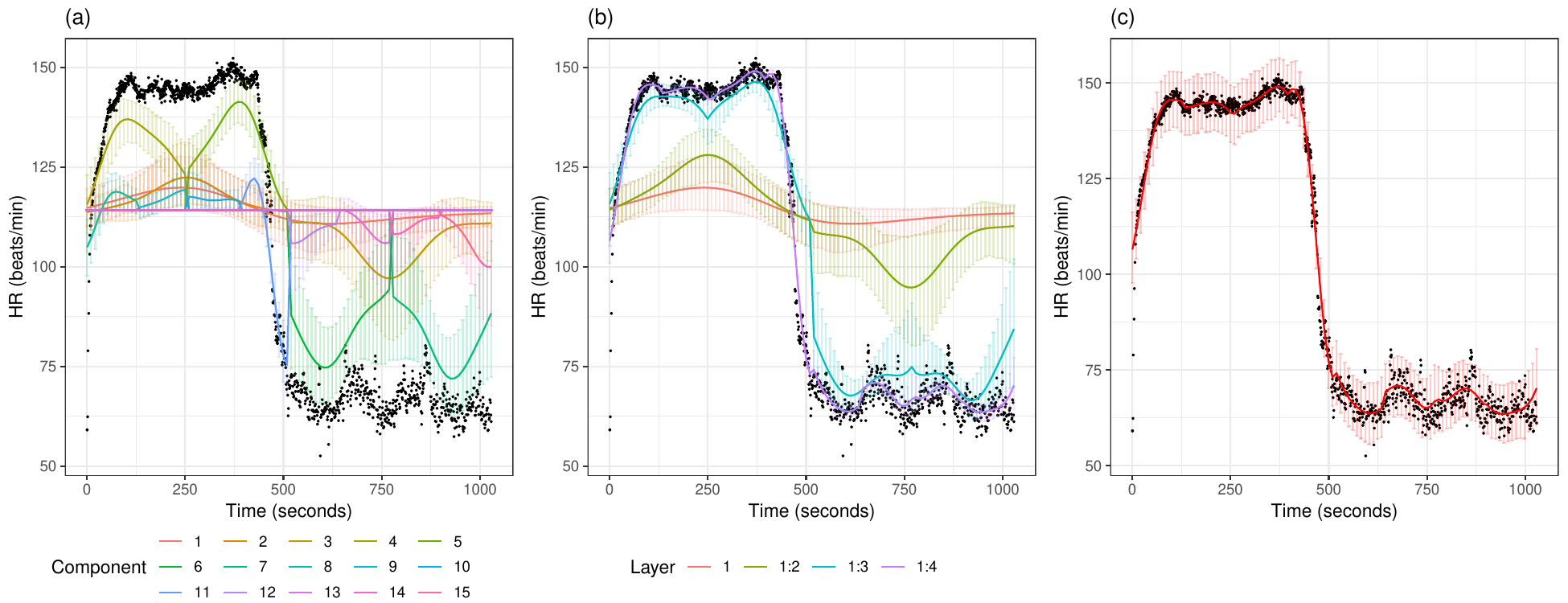}

\caption{The panels show the observed HR values over time as black dots and
results about the fit of the SAGP model with $m=5$ and $L=4$. Panel
(a) shows the posterior means and the 95\% CIs of the 15 additive
components of the SAGP model on 100 equispaced locations on the support
of the data. Panel (b) shows the posterior means and the 95\% CIs
of the sole component in layer 1 (red), of the components belonging
to layer 1 and 2 (green) , of the components belonging to layer 1,
2, 3 and of the complete model, including components from layer 1,
2, 3 and 4. Panel (c) provides the predictive mean and the corresponding
95\% prediction intervals. \label{fig:fit_HR-1}}
\end{figure}

\FloatBarrier



\vskip 0.2in 

\bibliography{bibliography}

\end{document}